\documentclass[11pt]{article}
\usepackage[english]{babel}							
\usepackage[utf8]{inputenc}
\usepackage{amsmath,amsfonts,amssymb,amsthm,cancel,mathtools,empheq,latexsym}
\usepackage{graphicx}
\usepackage{subfig,epsfig,tikz,float}		        
\usepackage{booktabs,multicol,multirow,tabularx,array} 
\usepackage{geometry}
\usepackage{bm}
\usepackage{xcolor}
\usepackage{mathrsfs}
\usepackage[normalem]{ulem}
\usepackage{hyperref}
\hypersetup{
    colorlinks=true,       % false: boxed links; true: colored links
    linkcolor=red,          % color of internal links (change box color with linkbordercolor)
    citecolor=blue,        % color of links to bibliography
    filecolor=magenta,      % color of file links
    urlcolor=cyan           % color of external links
}

\newtheorem{theorem}{Theorem}[section]

\newtheorem{corollary}{Corollary}[section]
\newtheorem{definition}{Definition}[section]
\newtheorem{proposition}{Proposition}[section]
\numberwithin{equation}{section}
\newcommand{\ave}[1]{ \left\{\!\!\left\{ {#1} \right\}\!\!\right\} }
\newcommand{\jump}[1]{ \left[ \! \left[ {#1} \right] \! \right] }
\newcommand{\triple}[1]{ |\!|\!|{#1}|\!|\!| }

\title{Boundary-field formulation for transient electromagnetic scattering by dielectric scatterers and coated conductors}

\author{%
George C. Hsiao$^{1}$ \ $\cdot$\
Tonatiuh S\'anchez-Vizuet$^{2*}$ \ $\cdot$ \
Wolfgang L. Wendland$^{3}$ }

\begin{document}
\date{}
\maketitle
\vspace{-0.75cm}
\begin{center}
{\footnotesize 
$^1$ University of Delaware \qquad
$^2$ The University of Arizona \qquad
$^3$ University of Stuttgart \\
{\scriptsize E-mails: {\tt ghsiao@udel.edu }/ {\tt tonatiuh@arizona.edu } / {\tt wolfgang.wendland@mathematik.uni-stuttgart.de } }\\
*Corresponding author\\[3ex]
\textit{Dedicated to Ernst P. Stephan, collaborator and friend on the occasion of his 75th birthday}
}
\end{center}

\begin{abstract}
We examine the transient scattered and transmitted fields generated when an incident electromagnetic wave impinges on a dielectric scatterer or a coated conductor embedded in an infinite space. By applying a boundary-field equation method, we reformulate the problem in the Laplace domain using the electric field equation inside the scatterer and a system of boundary integral equations for the scattered electric field in free space. To analyze this nonlocal boundary value problem, we replace it by an equivalent boundary value problem. Existence, uniqueness and stability of the weak solution to the equivalent BVP are established in appropriate function spaces in terms of the Laplace transformed variable. The stability bounds are translated into time-domain estimates which determine the regularity of the solution in terms of the regularity of the problem data. These estimates can be easily converted into error estimates for a numerical discretization on the convolution quadrature for time evolution.   
\end{abstract}

\medskip
\noindent
{\small{\bf Keywords}{:} 
Transient wave scattering, electrodynamics, time-dependent boundary integral equations, convolution quadrature.
}

\noindent
{\bf Mathematics Subject Classifications (2020)}: 78A40, 78M10, 78M15, 45A05.

% ===================================================
\section{Introduction}
% ===================================================
This paper is concerned with a basic scattering problem in time-dependent electromagnetics that can be simply described as follows: An incoming time-dependent electromagnetic wave impinges on a dielectric obstacle embedded in free space---such as air or vacuum. The objective of this paper is to examine the subsequent electromagnetic fields scattered by and transmitted into the dielectric obstacle by employing similar techniques to the ones developped by the authors in their recent paper on elastodynamic scattering \cite{HsSaWe:2022}. 

Electromagnetic scattering is a wide topic that has been studied for more than a century. The scientific literature in the electromagnetic scattering is abundant. However, most of the early work and contributions can be classified as what is called  {\em frequency-domain  scattering } in the sense that formulations are usually based on time-harmonic Maxwell's equations. For relevant references with respect to mathematical methods, boundary  integral equation  methods for direct and inverse scattering, optimization methods in electromagnetic radiation, low frequency  scattering, and numerical methods in electromagnetics,  see, for instance, Cessenat \cite{Ce:1996},  N\'{e}d\'{e}lec \cite{Ne:2001}, Stratton \cite{St:1941},  Colton \& Kress \cite{CoKe:1983} (updated as \cite{CoKe:1992}),  Martin \& Ola \cite{MaOl:1993}, Dassios \& Kleinman \cite{DaKl:2000}, Costabel, Darrigrand \&  Kon\'e\cite{CoDaKo:2010}, Angell \& Kirsch \cite{AnKi:2004}, Hsiao \& Kleinman \cite{HsKl:1997}, Hsiao, Monk \& Nigam \cite{HsMoNi:2002}, Stratton \& Chu \cite{StCh:1939}, M\"{u}ller \cite {Mu:1969},
Costabel \& Stephan \cite{CoSt:1988}, Kirsch.\cite{Ki:1989}, Stephan \cite{St:1986}, and  Poggio \& Miller \cite{PoMi:1973}, 
 to name a few. 

In  the monograph of Jones \cite[p. 44]{Jo:1964} and in the chapter of \cite[p.179]{PoMi:1973} by Poggio \& Miller,  electric  field  integral equations in space-time domain were derived by making use of the Fourier transform of the corresponding electric field integral equation in space - frequency  domain. Time-domain scattering has received renewed attention in the 21st. century. Time-domain integral equations were used in the papers  Weile, Pisharody, Chen, Shanker \& Michielssen \cite{WePiChShMi:2004}, and Pisharody \& Weile \cite{PiWe:2006}  for time- dependent electromagnetic scattering. In  Cakoni, Haddar \& Lechleiter  \cite{CaHaLe:2019}, the factorization method for a far field inverse scattering problem in the time domain is employed. In the recent book by Tong \& Chew \cite{ToCh:2020}, numerical solutions of time domain electric field and magnetic field equations are  presented. Since Convolution Quadrature (CQ) was introduced  and proved by Lubich \cite{Lu:1988, Lu:1994} to be an efficient numerical scheme for computing time-domain boundary integral equations, the study of time-domain scattering based on the Laplace transform approach has become very attractive and  popular  (see, e.g.  Chen, Monk, Wang \& Weile \cite{CMWW:2012}, Li, Monk \& Weile\cite{LiMoWe:2015},   Ballani, Banjai, Sauter \& Veit. \cite{BaBaSaVe:2013}, Banjai, Lubich \& Sayas, \cite{BaLuSa:2015},  Nick, Kov\'{a}cs \& Lubich\cite{NiKoLu:2022},  Banjai and Sayas\cite{BaSa:2022}). We believe that the chapter in the  {\em Encyclopedia of Computational Mechanics}  by  Costabel \& Sayas \cite{CoSa:2017} may serve as an excellent  introduction to the {\em time-domain scattering}  for new comers. For more advanced topics, interested readers can refer to Martin's recent monograph {\em Time-Domain  Scattering } \cite{Ma:2021}, which contains 910 references, and Sayas' book \textit{Retarded potentials and time domain boundary integral equations: a road-
map} \cite{Sa:2016errata,Sa:2016}.

The term {\em boundary-field formulation}, which is a variant of  {\em boundary-field equation method}, dates back to 1995 in the Research Notes \cite{GaHs:1995} by Gatica \& Hsiao for treating a class of elliptic linear and nonlinear problems in continuum  mechanics. The formulation  has been applied to space-frequency domain fluid structure interaction problems \cite {HsKlSc:1989}, and  to the space-frequency domain acoustic and electromagnetic scattering \cite{Hs:2000}.  In 2015, this boundary-field formulation was applied to the space-time domain fluid structure interaction problem by  Hsiao, Sayas \& Weinacht \cite{HsSaWe:2015}. The essential process of the formulation is as follows: first convert the space-time domain problem into a space-Laplace domain interaction problem; next, apply the boundary-field formulation to the resulting problem; then obtain appropriate estimates of the solutions in terms of the Laplace domain variable $s$, and finally obtain estimates of the corresponding solution in the space-time domain using a result due to Dominguez \& Sayas \cite{DoSa:2013} (improved and updated in \cite{Sa:2016errata,Sa:2016}). These techniques have been applied to a variety of time-domain problems in \cite{BrSaSa2018,HsSa:2021a,HsSaSa:2016, HsSaSaWe:2018, HsSa:2021b, HsWe:2021b, SaSa2016} including the recent  scattering problem in elastodynamics  by Hsiao, S\'anchez-Vizuet  \& Wendland \cite{HsSaWe:2022}.
 
The paper is organized as follows: The time-domain scattering problem is formulated in Section \ref{sec:2}. Section \ref{sec:3} is devoted to preliminaries  and notation including a brief refresher of trace spaces for the solution of Maxwell's equations, the space of causal tempered  distributions for the Maxwell problem and their Laplace transform. Section \ref{sec:4} lays out the building blocks for the transformation of the Laplace domain system into a boundary-field formulation, which is then proposed and analyzed in Section \ref{sec:5}. The core of the analysis is contained in Section \ref{sec:6} which includes existence, uniqueness and stability results. These are achived via the introduction of an equivalent transmission problem. The stability estimates in the Laplace domain summarized in Theorem \ref{thm:6.1} and Corollaries \ref{cor:cor6.1} and \ref{cor:cor6.2} are then transferred into time domain estimates in Section \ref{sec:7}. Finally, in Section \ref{sec:8} we introduce the setting of a coated conducting scatterer and briefly show how all the results from the previous sections can be applied verbatim to this more general situation.
%
%=======================================
\section{The scattering problem}\label{sec:2}
%==========================================
%
%=================================
\subsection{Formulation of the problem} 
%==================================
Let $\Omega_-\subset  \mathbb R^3$ be a bounded domain representing an electrically neutral dielectric obstacle, the scatterer, which is assumed to have a Lipschitz  boundary $ \Gamma: = \partial \Omega_ - $ with an 
outward unit normal ${\bf n}$ as shown in Figure \ref{fig:fig1}. 
\begin{figure}[tb] \centering 
\includegraphics[width=0.4\linewidth]{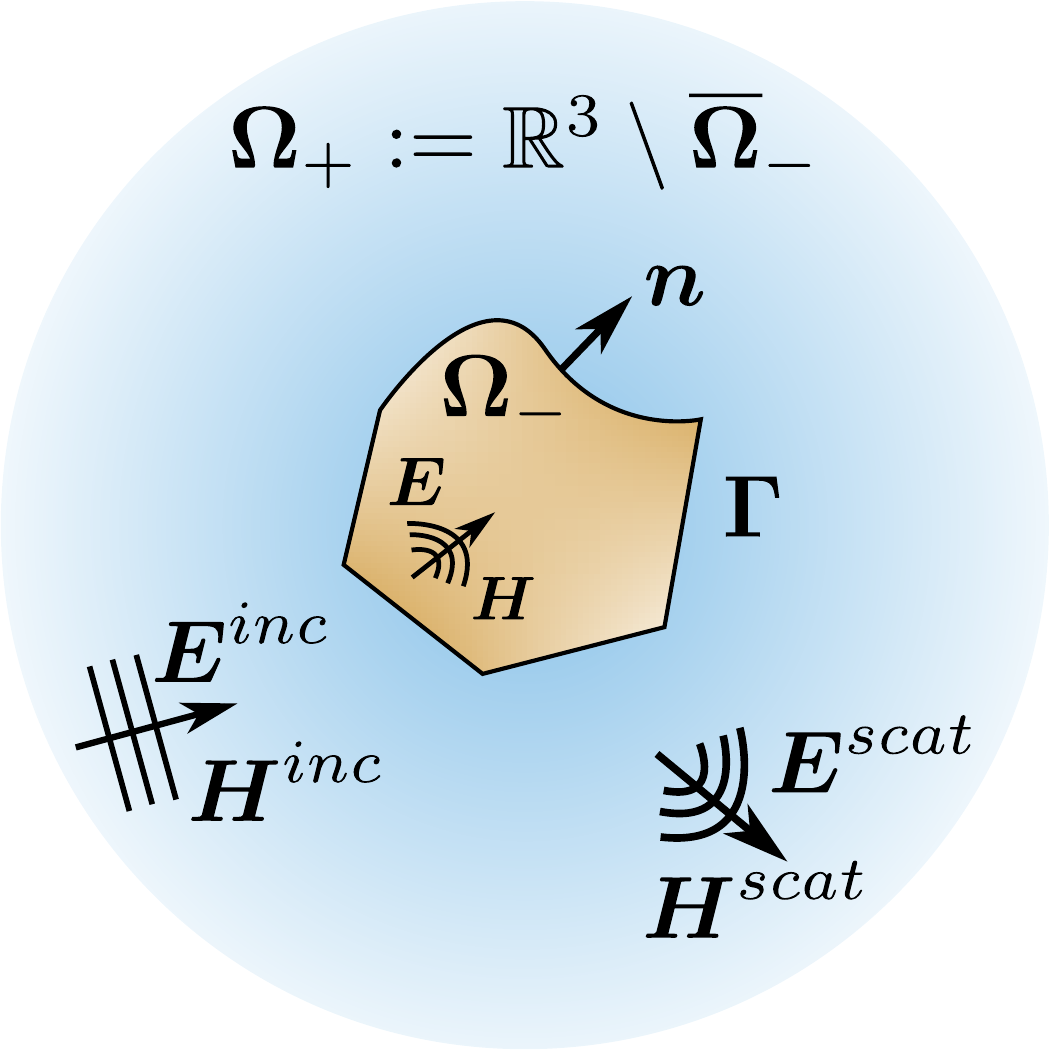}
\caption{  A schematic depiction of the problem geometry. }\label{fig:fig1} 
\end{figure} 
We begin the formulation of the scattering problem with the governing equations inside the dielectric obstacle $\Omega_-$ consisting  of the time-dependent system of Maxwell's equations
 \begin{equation}
\left.
\begin{array}{rclcrcl}
{\bf curl} \, {\bf E} & = & -\partial_t\,{ \bf B} , & \quad & {\bf curl}\,{\bf H} & = & \partial_t\,{\bf D} + {\bf J},
\\ 
{\rm div}\,  {\bf D} & = & 0, & \quad & {\rm div}  \,{\bf B} & = & 0  
\end{array}
\right\} \quad \mbox{in} \quad \Omega_- \times (0, T)  \label{eq:h2.1}
\end{equation}
for the electric and magnetic fields $\bf E$ and $\bf H$, respectively. The vector valued function ${\bf J} $ describes an external applied current density which satisfies the compatibility condition
\[
{\rm div} {\bf J} = 0 \quad \mbox{in} \quad   \Omega_- \times (0, T).
\]
The vectors ${\bf D}$ and ${\bf B}$ are referred to as the electric displacement and magnetic induction, respectively. They are related to the electric and magnetic fields by the standard constitutive equations for linear media 
\[
{\bf D} = \varepsilon {\bf E}, \quad \text{ and } \quad {\bf B} = \mu {\bf H} ,
\]
where $\varepsilon $ is the electric permittivity and $\mu$  is the magnetic  permeability of the medium. For simplicity, we assume that  both of them are positive constants. \color{black} However, as long as the product $\epsilon\mu\in L^\infty(\Omega_-)$ and there exists some $k_0>0$ such that $\epsilon\mu\geq k_0$, the analysis technique and the results will remain valid. \color{black} When the medium is vacuum, it is custumary to denote these constants by $\epsilon_0$ and $\mu_0$. We consider the free space exterior to the scatterer as vacuum and denote it by $\Omega_+ := \mathbb{R}^3 \setminus \overline{\Omega}_-$. Here, as usual, we decompose the 
total electric ${\bf E}$ and magnetic ${\bf H}$  fields in the form:
\[
{\bf E} := {\bf E}^{scat} + {\bf E}^{inc} , \quad  {\bf H} := {\bf H}^{scat} + {\bf H}^{inc}
\]
for given incident fields $ {\bf E}^{inc}$ and $ {\bf H}^{inc} $. The scattered fields ${\bf E}^{scat} $ and ${\bf H}^{scat}$
are required to satisfy the  corresponding  time-dependent Maxwell's equations (see, for instance, Hsiao and Wendland \cite{HsWe:2021}, and Monk \cite{Monk:2003}):
\begin{equation}
\left.
\begin{aligned}
{\bf curl} \, {\bf E}^{scat} + \mu_0 \, \partial_t\,{\bf H}^{scat} & = \boldsymbol 0 \\
 {\bf curl}\,{\bf H}^{scat} - \varepsilon_0\, \partial_t\,{\bf E}^{scat} & = \boldsymbol 0 
\end{aligned}
\right\} \quad \mbox{in} \quad \Omega_+ \times (0, T). \label{eq:h2.2}
\end{equation} 
Together with \eqref{eq:h2.1} and \eqref{eq:h2.2} we require the electric and magnetic fields to satisfy the transmission and initial conditions
\begin{equation}
\left.
\begin{aligned}
{ \bf n} \times ({ \bf E} - {\bf E}^{scat})\times {\bf n} & = {\bf n} \times {\bf E}^{inc}\times {\bf n} \\
{\bf n} \times ({ \bf H} - {\bf H}^{scat})\times {\bf n} & =  {\bf n} \times 
 {\bf H}^{inc}\times {\bf n}
\end{aligned}
  \right\} \quad  \mbox{on} \quad \Gamma \times (0,  T) \label{eq:h2.3}
  \end{equation}
 and
 \begin{equation}
  \label{eq:h2.4}
  {\bf E}(x, 0) =  {\bf 0}  \quad{\rm in}\quad   \Omega_-  \quad \mbox{and} \quad  
{\bf E}^{scat}(x, 0)  =  {\bf 0} \quad  \mbox{in} \quad \Omega_+.
\end{equation} 
 We suppose that ${\bf E}^{inc}$ and ${\bf H}^{inc}$ are given incident fields satisfying  \eqref{eq:h2.1} and \eqref{eq:h2.2} in free space. Across the interface $\Gamma$ the tangent components of both ${\bf E}$  and ${\bf H}$ must be continuous. In addition, we assume homogeneous initial conditions for both ${\bf E} $ and ${\bf H} $, and  require that the given incident fields ${\bf E}^{inc} $  and ${\bf H}^{inc}$ are causal in the sense that they vanish identically  for $ t <  0$.
%
% ==================================================
\subsection{Governing equations in the time domain} 
% ==================================================
%
It is a standard procedure to eliminate one variable in \eqref{eq:h2.1} and \eqref{eq:h2.2} and express the problem in terms of a second order system. Eliminating the fields ${\bf H}$ and ${\bf H}^{scat}$ the time-dependent scattering problem defined by \eqref{eq:h2.1},\eqref{eq:h2.2}, \eqref{eq:h2.3} and \eqref{eq:h2.4} can be simplified into: Given a divergence-free vector field ${\bf J}$ and an incident field ${\bf E}^{{inc}}$, find  vector fields ${\bf E} $ and ${\bf E}^{scat}$ such that
 \begin{equation}
 \label{eq:h2.5}
\left.\begin{array}{rclcr}
{\bf curl} \,{\bf curl}\,  {\bf E} + c^{-2} \, \partial^2_{tt}\,{\bf E}  & = & - \mu \partial_t {\bf J} & \quad & \text{ in } \quad \Omega_- \times (0, T) \\
 {\bf curl}\, {\bf curl}\, {\bf E}^{scat} + c_0^{-2} \, \partial^2_{tt} \,{\bf E}^{scat} & = &  {\bf 0} & \quad & \text{in } \quad \Omega_+ \times (0, T) 
\end{array}\right\}.
\end{equation} 
Where the constants $c: = 1/ \sqrt{\mu \varepsilon} $ and $ c_0 := 1/ \sqrt{ \mu_0 \varepsilon_0}$ represent the speed of light in the dielectric and in vacuum respectively. These equations are supplemented with the transmission conditions
\begin{equation}
\label{eq:h2.6}
\left.
\begin{array}{rcl}
{ \bf n} \times ({ \bf E} - {\bf E}^{scat})\times {\bf n}  & = & {\bf n} \times {\bf E}^{inc} \times {\bf n} \\
{\bf n} \times { \bf curl} \, ( \mu^{-1}{\bf E} - \mu_0^{-1}{\bf E}^{scat})  & = &  {\bf n} \times 
\mu_0^{-1}{\bf curl }\, {\bf E}^{inc}
\end{array}\, 
  \right\} \quad \mbox{on} \quad \Gamma \times (0, T),  
\end{equation}
and initial conditions
\begin{alignat}{8}
\nonumber
{\bf E}(x, 0) =&\, {\bf 0},  &\quad &  \partial_t {\bf E} (x, 0)  &=\, {\bf 0}  & \quad &  \text{ in }\, \Omega_-\,, \\
\label{eq:h2.7}
{\bf E}^{scat}(x, 0) =&\,  {\bf 0},  &\quad &  \partial_t {\bf E}^{scat}(x, 0) &=\, {\bf 0} & \quad &  \text{ in }\, \Omega_+\,.
\end{alignat}
In the above formulation,  the transmission conditions \eqref{eq:h2.3} have been replaced by \eqref{eq:h2.6}. Moreover, the restrictions of the functions to $\Gamma$ should be understood in the sense of traces. This  will be treated with more detail in Section \ref{sec:3}, where we introduce the traces of the solutions of Maxwell's equations.
%
%===================================================
\section{Preliminaries and notation}\label{sec:3} 
%===================================================
%
Before presenting the boundary-field formulation of the Laplace-transformed scattering problem, we must first briefly introduce some notation and properties of the solution spaces of Maxwell's equations. The full development of this theory is relatively recent and quite technical. Therefore we present only the basics for completeness and refer the reader to the seminal articles \cite{BuCoSc:2002, BuCoSh:2002} and to the monographs \cite{HsWe:2021,Monk:2003,Ne:2001} for the full details. Throughout this communication we will make use of standard results and notation on Sobolev space theory, the basics of which can be found in the classic text \cite{AdFo:2003}.  
 
% =====================================================
\subsection{Sobolev spaces for the curl}
% ====================================================

Let $\mathcal{O}$ be a bounded Lipschitz domain in $\mathbb{R}^3$ with boundary $\Gamma:=\partial \mathcal{O}$. The notation
\[
({\bf u},{\bf v})_{\mathcal O} := \int_{\mathcal O} {\bf u}\cdot{\bf v} \quad \text{ and } \quad \langle{\bf u},{\bf v}\rangle_\Gamma := \int_\Gamma {\bf u}\cdot{\bf v},
\]
will be used to denote the inner product between the vector-valued functions ${\bf u}$ and ${\bf v}$ defined in the domain and on its boundary, respectively. The function space 
 \[
 { \bf H} (\mathrm{{\bf curl}}, \mathcal{O}) :=  \{ {\bf u \in L^2} (\mathcal{O}) := \mathrm L^2( \mathcal{O})^3: \mathrm{{\bf curl}}\,\mathbf u \in {\bf L}^2 (\mathcal{O}) \}
 \]
endowed with the inner product
$
(\mathbf u,\mathbf v)_{{\bf curl},\mathcal O} := \int_{\mathcal O} (\mathbf u \cdot \mathbf v + {\bf curl\, u \cdot curl\, v})
$
becomes a Hilbert space with the induced norm 
 \[
 \|{\bf u} \|^2 _{\mathrm{{\bf curl}},\; \mathcal{O}} := (\mathbf u,\overline {\mathbf u})_{{\bf curl},\mathcal O} = \int_{\mathcal O} (\mathbf u \cdot \overline {\mathbf u} + {\bf curl\, u \cdot curl\, \overline {u}}) = \| {\bf u}\|^2_{L^2(\mathcal{O})}+ \| \mathrm{{\bf curl}}\,{\bf  u} \|^2_{L^2(\mathcal{O})} .
 \]
 Above, $\overline{\bf u}$ is used to denote the complex conjugate of the function ${\bf u}$.
  
We recall that, in Sobolev space theory, the operator mapping a function defined over a domain $\mathcal O$ to its restriction to the boundary $\Gamma$ (whenever this operation is well defined) is usually denoted by $\gamma$ and referred to as the \textit{trace operator}. When applied to a vector-valued function, the trace should be understood as acting componentwise. Moreover, if the domain $\mathcal O$ is bounded (as will be the case in our application) the notation $\gamma^{-}$ will be used to denote the trace in the interior domain $\mathcal O_-:=\mathcal O$, while $\gamma^+$ will denote the trace in the exterior domain $\mathcal O_+ := \mathbb R^3\setminus \overline{\mathcal O}$. This superscript convention will extend to the tangential traces that will be defined later. Similarly, for a vector valued function $\mathbf u$ defined over $\mathbb R^3\setminus \Gamma$ and its external and internal traces (or tangential traces) we will use the standard notation
\[
\jump{\gamma \mathbf u} : = \gamma^-\mathbf u - \gamma^+\mathbf u \qquad \text{ and } \qquad \ave{\gamma \mathbf u} : = \frac{1}{2}\left(\gamma^-\mathbf u + \gamma^+\mathbf u \right)
\]
to denote, respectively, its \textit{jump}  and \textit{average} across the interface $\Gamma$.

For a vector-valued function ${\bf u}$ defined on the surface $\Gamma$, the operators $ {\rm div}_{\Gamma}$ and ${\rm curl}_{\Gamma} $ will denote the surface divergence and surface scalar curl defined respectively by
\[
{\rm div}_{\Gamma} \, {\bf u} := \gamma({\rm div} \, {\bf \tilde{u}})\quad \text{ and } \quad {\rm curl}_{\Gamma} \ {\bf u} := {\bf n} \cdot \gamma({\bf curl \, \tilde{u}}),
\]
where ${\bf \tilde{u}} $ is an extension of ${\bf u}$ defined on the tubular neighborhood of $\Gamma$ in $\mathcal{O}$ and such that $\gamma {\bf \tilde u}={\bf u}$. 
 
In the case of functions ${\bf u \in H} (\mathrm{{\bf curl}}, \mathcal{O})$, it is possible to define two boundary restriction operators or traces. The \textit{tangential trace operator} $\gamma_t$, mapping a vector field into its component tangential to $\Gamma$; and the \textit{tangential projection operator} $\pi_t$ mapping a vector field to the projection of its componentwise trace into the tangent bundle to $\Gamma$ \cite{HsWe:2021,Ne:2001}. These two restrictions are not equal to each other (for instance, their orientation in the tangent bundle differs by ninety degrees) and are defined by
 \[
\gamma_t\mathbf u : = \boldsymbol n \times \gamma\mathbf u \qquad \text{ and } \qquad \pi_t \mathbf u := \boldsymbol n \times (\gamma{\bf u} \times \boldsymbol n) = \gamma\mathbf u - (\gamma\mathbf u\cdot \boldsymbol n)\boldsymbol n.
\]
Above, the componentwise trace of $\mathbf u$ is well defined since the space $\mathbf H^1(\mathcal{O}):=\left(\mathrm H^1(\mathcal{O})\right)^3$ is dense in  $\mathbf H(\mathrm{{\bf curl}}, \mathcal{O})$. In the paramount article \cite{BuCoSh:2002} it was shown that, for Lipschitz domains, these operators are bounded, linear and surjective over the function spaces 
 \[
 \gamma_t : { \bf H} (\mathrm{{\bf curl}}, \mathcal{O}) \rightarrow {\bf H}_{\|}^{-1/2}(\mathrm{div}_{\Gamma} , \Gamma) 
 \quad \mbox{and}\quad
 \pi_t : 
  { \bf H} (\mathrm{ {\bf curl}}, \mathcal{O}) \rightarrow {\bf H}_{\perp}^{-1/2}(\mathrm{curl}_{\Gamma}, \Gamma). 
\]
The definition of the tangential trace spaces ${\bf H}_{\|}^{-1/2}(\mathrm{div}_{\Gamma} , \Gamma) $ and ${\bf H}_{\perp}^{-1/2}(\mathrm{curl}_{\Gamma},\Gamma)$ is elaborate and we shall omit the details here. Let us simply say that these spaces are respectively endowed with the graph norms $ \| \cdot \|^2_{-1/2, \mathrm{{\rm div}_{\Gamma}}, \Gamma}$ and $ \| \cdot \|^2_{-1/2, \mathrm{{\rm curl}_{\Gamma}}, \Gamma}$, and that the space  ${\bf H}_{\|}^{-1/2}(\mathrm{div}_{\Gamma} , \Gamma)$ is naturally identified with the dual space of ${\bf H}_{\perp}^{-1/2}(\mathrm{curl}_{\Gamma},\Gamma)$,  when we use
\[
{\bf L}^2_t(\Gamma):= \left\{{\bf u} \in {\bf L}^2(\Gamma): {\bf u}\cdot\boldsymbol n = 0 \; \text{a.e. on }\Gamma\right\}
\]
as ``pivot" space.  The reader is referred to \cite[\S 4.5]{HsWe:2021} and \cite{BuCoSh:2002} for a detailed discussion.

As a consequence of the surjectivity of the tangential traces, the pseudo-inverses 
\[
\gamma_t^\dagger:\mathbf H^{-1/2}_{\|}({\rm div}_\Gamma,\Gamma) \to \mathbf H({\bf curl},\mathcal O) \qquad \text{ and } \qquad \pi_t^\dagger: \mathbf H^{-1/2}_{\perp}({\rm curl}_\Gamma,\Gamma) \to \mathbf H({\bf curl},\mathcal O),
\]
of these mappings are well defined and bounded. We shall refer to these pseudo-inverses as \textit{liftings}. Using the liftings and the definitions of the tangential traces it is easy to see that
\[
 \gamma_t {\bf u} = \pi_t (\gamma_t {\bf u}) = \gamma_t ( \pi_t {\bf u} ),\quad \text{ and } \quad  \pi_t (\pi_t {\bf u}) = \pi_t {\bf u}
\]
for any  $\mathbf u \in \mathbf H(\mathrm{{\bf curl}}, \mathcal{O})$, where the successive application of two trace operators should be understood as being mediated by the corresponding lifting, which has been omitted for simplicity. We shall use the identities above---including the implicit use of liftings---liberally throughout this work.

Let $\mathcal O_-:=\mathcal O$ and $\mathcal O_+:= \mathbb R^3\setminus\overline{\mathcal O}$ denote respectively the bounded and unbounded components of $\mathbb R^3$ separated by $\Gamma:=\partial\mathcal O$.

 In the sequel we will make heavy use of the following integration by parts formulas, or Green's formulas, which hold for ${\bf u},\,{\bf v}\in \mathbf H({\bf curl}, \mathcal O_-) $ and  ${\bf u},\,{\bf v}\in \mathbf H({\bf curl}, \mathcal O_+) $ respectively:
\begin{subequations}
\label{eq:integrationbyparts}
\begin{align}
\label{eq:ibpInt}
\langle\gamma_t^-{\bf u},\pi_t^-{\bf v}\rangle_\Gamma &= ({\bf curl\, u},{\bf v})_{\mathcal O_-}  - ({\bf u},{\bf curl\, v})_{\mathcal O_-}\\
\label{eq:ibpExt}
\langle\gamma_t^+{\bf u},\pi_t^+{\bf v}\rangle_\Gamma &= ({\bf u},{\bf curl\, v})_{\mathcal O_+} - ({\bf curl\, u},{\bf v})_{\mathcal O_+}.
\end{align}
\end{subequations}
The validity of these formulas in $\rm H({\bf curl}, \Omega)$ when $\Omega$ has only Lipschitz regularity (proved in \cite{BuCoSh:2002}) is highly non-trivial, and is one of the main results of the theory developped by Annalisa Buffa and her collaborators in the early 2000's.
%
% ==============================
\subsection{Causal tempered distributions for Maxwell's problem}
% ==============================
In view of the structure of equations \eqref{eq:h2.5}, we start by defining the somewhat unusual space
\[
\mathbf H({\bf curl}^2, \mathbb R^3\setminus\Gamma) := \left\{ \mathbf u \in \mathbf H({\bf curl}, \mathbb R^3\setminus\Gamma): {\bf curl\,curl\,u}\in{\bf L}^2(\mathbb R^3\setminus\Gamma) \right\},
\] 
which becomes a Hilbert space when endowed with the norm
\[
\|{\bf u}\|_{\bf curl\,curl}^2 := \|{\bf u}\|_{\mathbb R^3\setminus\Gamma}^2 + \|{\bf curl\, u}\|_{\mathbb R^3\setminus\Gamma}^2 + \|{\bf curl\,curl\, u}\|_{\mathbb R^3\setminus\Gamma}^2.
\]
In the definition above, the subscript $\mathbb R^3\setminus\Gamma$ has been used to emphasize that both ${\bf curl\, u}$ and ${\bf curl\,curl\, u}$ are meant in the sense of distributions over $\mathbb R^3\setminus\Gamma$. 

If we denote by $\mathcal P$ the space of polynomials, the Schwartz class is defined as
\[
\mathcal S(\mathbb R) :=\{\varphi \in \mathcal C^\infty(\mathbb R): p\varphi^{(k)} \in L^\infty(\mathbb R) \,\forall k\geq 0, \forall p\in\mathcal P\}.
\]
A continuous linear map $T: \mathcal S(\mathbb R)\to \mathbf H({\bf curl}^2, \mathbb R^3\setminus\Gamma) $ such that $T(\varphi)=\boldsymbol 0$ for all $\varphi\in \mathcal S(\mathbb R)$ with $\text{ supp }\varphi\subset (-\infty,0)$ is said to be a \textit{causal tempered distribution with values in } $\mathbf H({\bf curl}^2, \mathbb R^3\setminus\Gamma)$ and is denoted as $T\in \text{CT}(\mathbf H({\bf curl}^2, \mathbb R^3\setminus\Gamma))$. 

It then makes sense to consider distributional solutions of Maxwell's equations understood as searching for ${\bf E}$ and ${\bf E}^{scat}$ in $\text{CT}(\mathbf H({\bf curl}^2, \mathbb R^3\setminus\Gamma))$ such that equations \eqref{eq:h2.5} and  \eqref{eq:h2.6} hold for given causal tempered distributions ${\bf J}$ and ${\bf E}^{inc}$. Moreover, the causal initial conditions  \eqref{eq:h2.7}, as we shall see next, will also allow us to consider the distributional Laplace transform of the system.
% =========================================
\subsection{Laplace transform}
% =========================================
%
In order to apply the boundary-field equation method to the problem at hand,  we will consider its formulation in the Laplace domain. Throughout the paper we will denote the  complex plane by $\mathbb{C}$; we will refer to the set of complex numbers with positive real part as \textit{the positive half-plane} and will denote it by $ \mathbb{C} + := \{ s \in  \mathbb{C}  : \mathrm{Re}\, s > 0 \}$. We will also repeatedly make use of the notation
\[
\sigma : = \mathrm{Re}(s) \quad\text{ and }\quad \underline{\sigma}: = \min\{1,\sigma\}.
\]
For any causal complex-valued function with limited growth at infinity $ f : [0,\infty) \rightarrow  \mathbb{C}$ , its Laplace transform is given by
\[
\mathcal L \{f\} (s) :=  \int_0^{\infty} e^{-st} f(t)\, dt,  
\]
whenever the integral converges. If $X$ is a Banach space and $f\in {\rm CT}(X)$, the Laplace transform of $f$ can be defined by duality as the analytic function
\[
\mathcal L \{f\} := f\left(\text{exp}(-s\cdot)\right) \quad \text{ for } s\in\mathbb C_+.
\]
Hence, for a causal tempered distribution ${\mathbf u}\in \text{CT}(\mathbf H({\bf curl}^2, \mathbb R^3\setminus\Gamma))$, the Laplace transform is well defined. The reader is referred to \cite[Chapter 2]{Sa:2016} for further details regarding distributional Laplace transforms.

% ===========================
\subsection{Notational remarks}
% ===========================
Since most of the analysis that follows will be carried out in the Laplace domain, and to keep the notation as simple as possible we will drop the explicit use of ``\,$\mathcal L \{\cdot\}$\," to denote the Laplace transform. We will assume implicitly that all the variables are the Laplace transforms of a corresponding function in the physical space that is denoted with the same symbol. When the time comes to translate our results into the physical domain, we will indicate so explicitly.

In a similar vein, to avoid the proliferation of superfluous constants in our estimates, we will make use of the symbols $a \gtrsim b$ and $a \lesssim  b$  to denote the existence of a generic positive constant $C$ \textit{independent of the Laplace parameter s} or of any relevant physical parameters such that  $a \geq C  b $ and  $a \leq C  b$ respectively.
 
%
% =====================================================
\section{Governing equations in the Laplace domain}\label{sec:4}
% =====================================================
%
Having introduced the Laplace transform for causal tempered distributions, we can now consider the Laplace domain version of the distributional system \eqref{eq:h2.5} through \eqref{eq:h2.7}, which becomes:
\begin{subequations}
\label{eq:h2.8}
\begin{alignat}{6}
\label{eq:h2.8a}
{\bf curl}\,{\bf curl}\, {\bf E} + (s/c)^2\, {\bf E}  &= - \mu \, s {\bf J}  &\quad& \text{ in } \Omega_-\,,\\
\label{eq:h2.8b} 
{\bf curl}\,{\bf curl}\, {\bf E}^{scat} + (s/c_0)^2\, {\bf E}^{scat}  &=  {\bf 0}  & \quad & \text{ in } \Omega_+\,,\\
\label{eq:h2.8c}
{ \bf n} \times ({\bf E} - {\bf E}^{scat}) \times {\bf n}  & = {\bf n} \times {\bf E}^{inc} \times {\bf n} & \quad & \text{ on } \Gamma, \\
\label{eq:h2.8d}
 {\bf n} \times { \bf curl}\,(\mu^{-1}{\bf E} - \mu_0^{-1}{\bf E }^{scat}) &   = {\bf n} \times \mu_0^{-1}{\bf curl } \, {\bf E}^{inc} & \quad & \text{ on } \Gamma.
\end{alignat}
\end{subequations}
In this section, we will transform this Laplace-domain system of PDEs into a non-local problem consisting of boundary integral equations for the scattered field in $\Omega_+$ and a variational formulation for the transmitted field in $\Omega_-$.
%
% ==================================================
\subsection{Operator equations in $\Omega_-$}
% ==================================================
%
For a complex number $s$ and an open domain $\mathcal O$ with Lipschitz boundary, we define the parameter-dependent bilinear form $a(\cdot,\cdot;\cdot)_{\mathcal O}: {\bf H}({\bf curl},\mathcal O)\times {\bf H}({\bf curl},\mathcal O)\times \mathbb C \to \mathbb C$ by
\begin{equation}
\label{eq:h3.2}  
a ({\bf u},  {\bf v}; s)_{\mathcal O} : =  ({\bf curl}\, {\bf u}, {\bf curl\, v})_{\mathcal O} + s^2 ({\bf u}, {\bf v})_{\mathcal 
O},  \quad 
\forall \quad {\bf u}, 
{\bf v} \in {\bf H}({\bf curl}, \mathcal O).  
\end{equation}
Moreover, for ${\bf u}  \in\,  {\bf H}({\bf curl}, \mathcal O)$ we define its associated energy norm
\begin{equation}
\label{eq:energynorm}
\triple{{\bf u }}^2_{|s|, {\bf curl}, \mathcal O} := \|{\bf curl\, u} \|^2_{\mathcal O}  + \| s {\bf u} \|^2_{\mathcal O},
\end{equation}
which is easily seen to be a norm and to satisfy the equivalence relations
\begin{equation}
\label{eq:energynormP1}
\underline{\sigma}\, \triple{{\bf u }}_{|1|, {\bf curl}, \mathcal O}  \leq \triple{{\bf u} }_{|s|, {\bf curl}, \mathcal O} \leq \frac{|s|}{\underline{\sigma}} \triple{{\bf u }}_{|1|, {\bf curl}, \mathcal O}.
\end{equation}

Hence, without loss of generality, for $s =1 $ we will assume that the norm  $\triple{\cdot}_{|1|, {\bf curl}, \mathcal O} $ is equivalent to the norm $ \| \cdot \|_{{\bf curl}, \mathcal O} $ of ${\bf H}({\bf curl}, \mathcal O) $. It is also easy to prove that
\begin{align}
\label{eq:energynormP2}
|a( {\bf u} , {\bf v};  s )_{\mathcal O} | & \leq \triple{{\bf u }}_{|s|, {\bf curl}, \mathcal O}
\triple{{\bf v }}_{|s|, {\bf curl}, \mathcal O},\\
\label{eq:energynormP3}
\mathrm{Re} \Big(\bar{s} \, a( {\bf u} , \overline{\bf{u}}; s )_{\mathcal O}  \Big) & = \sigma \triple{{\bf u} }^2_{|s|, {\bf curl}, \mathcal O}.
\end{align}
These continuity and coercivity bounds will be key in proving the well posedness of the problem at hand.

For ${\bf E}  \in\,  {\bf H}({\bf curl}, \Omega_-)$, multiplying \eqref{eq:h2.8a} by the test function ${\bf F} \in 
{\bf H}( {\bf curl}, \Omega_-)$ and applying Green's formula \eqref{eq:ibpInt} we obtain the variational form of the equation \eqref{eq:h2.8a}:
\begin{equation}
\label{eq:h3.1}
a ({\bf E},  {\bf F}; s/c)_{\Omega_-}  + \langle \gamma_t^- {\bf curl\, E}, \pi_t^-{ \bf F }\rangle_{\Gamma} =   -s \mu \, ( {\bf J}, {\bf F })_{\Omega_-},  
\end{equation}
which can also be written in operator form as
\begin{equation} 
\label{eq:h3.3}
  {\bf A}_{\Omega_-} (s/c) {\bf E} + \left(\pi_t^-\right)^\prime\gamma_t^-{\bf curl\, E}= - s \mu \, {\bf J}\quad \mbox{in} \quad  {\bf H}({\bf curl}, \Omega_- )^{\prime},
  \end{equation}
where, for an open domain $\mathcal O$, the operator ${\bf A}_{\mathcal O}$ mapping ${\bf H}({\bf curl},\mathcal O)$ into its dual space ${\bf H}^\prime({\bf curl},\mathcal O)$ has been defined as
\[
  \langle{\bf A}_{\mathcal O} (s) {\bf E},{\bf F}\rangle  := a ({\bf E},  { \bf F }; s)_{\mathcal O} \quad \text{ for } {\bf F} \in {\bf H}({\bf curl},\mathcal O).
\]
%
% =========================================================
\subsection{Electromagnetic layer potentials and boundary integral operators}
% =========================================================
We now recall the definitions of the electromagnetic layer potentials and the boundary integral operators as well as the properties that will be used to reduce the exterior problem into a non-local problem over $\Gamma$. For $\mathbf m\in {\bf H}^{-1/2}_{\perp} ( {\rm curl}_{\Gamma}, \Gamma)$ and ${\bf j} \in {\bf H}^{-1/2}_{||} ({\rm div}_{\Gamma}, \Gamma )$, we define the electromagnetic layer potentials
\begin{alignat*}{6}
\nonumber
\mathcal D(s){\bf j} :=\,&{\bf curl}\!\! \int_\Gamma \!\!G\left(x,y;s\right){\bf j}(y)dy\,, &\qquad&& \mathcal S(s){\bf j}:=\,& {\bf curl\,curl}\!\!\int_\Gamma\!\! G\left(x,y;s\right){\bf j}(y)dy,\\
\widetilde{\mathcal D}(s)\mathbf m :=\,& \mathcal D(s)\circ\gamma_t\circ\pi_t^\dagger\,\mathbf m\,, &\qquad&& \widetilde{\mathcal S}(s)\mathbf m :=\,& \mathcal S(s)\circ\gamma_t\circ\pi_t^\dagger\mathbf m\,,
\end{alignat*} 
and the boundary integral operators
\begin{alignat*}{6}
\nonumber
\mathcal K(s) \mathbf j :=\,& \gamma_t{\bf curl} \int_\Gamma G(x,y;s){\bf j}(y)dy, & \qquad && \mathcal V(s) {\bf j} :=\,& \pi_t{\bf curl\,curl}\int_\Gamma G(x,y;s){\bf j}(y)dy, \\
 \widetilde{\mathcal K}(s) \mathbf m :=&\, \mathcal K(s)\circ\gamma_t\circ\pi_t^\dagger{\bf m},  & \qquad && \widetilde{\mathcal V}(s){\bf m}:=\,&\mathcal V(s)\circ\gamma_t\circ\pi_t^\dagger{\bf m}.
\end{alignat*}
The traces of the electromagnetic layer potentials satisfy
\begin{equation}
\label{eq:LPtraces}
\gamma_t^{\pm}\mathcal D(s) = \pm\frac{1}{2} + \mathcal K(s)
\quad \text{ and } \quad  
 \pi_t^{\pm}\widetilde{\mathcal D}(s) = \pm\frac{1}{2} + \widetilde{\mathcal K}(s),
\end{equation}
which in turn imply the well known average conditions
\begin{alignat*}{6}
\nonumber
\ave{\gamma_t\mathcal D(s){\bf j}} =\,& \mathcal K(s){\bf j}\,, & \qquad \qquad \qquad && \ave{\pi_t \mathcal S(s){\bf j}} =\,& \mathcal V(s){\bf j}\,, \\
\ave{\pi_t\widetilde{\mathcal D}(s){\bf m}} =\,& \widetilde{\mathcal K}(s){\bf m}\,, & \qquad \qquad \qquad && \ave{\gamma_t\widetilde{\mathcal S}(s){\bf m}} =\,& \widetilde{\mathcal V}(s){\bf m},
\end{alignat*} 
and the jump conditions
\[
\jump{\gamma_t\mathcal D(s){\bf j}\,} = -{\bf j}\,, \quad \jump{\pi_t\widetilde{\mathcal D}(s){\bf m}} = -{\bf m}\,, \quad \text{ and } \quad \jump{\pi_t \mathcal S(s){\bf j}\,} =\jump{\gamma_t\widetilde{\mathcal S}(s){\bf m}} = {\bf 0}.
\]
%
%========================================
\subsection{Reduction to BIEs in $\Omega_+$}
%========================================
In order to transform the partial differential equation \eqref{eq:h2.8b} into a BIE on the boundary of $\Omega_+$, we must first find the representation of the solution of \eqref{eq:h2.8b} in terms of Cauchy data 
\[
{\bf m} \in {\bf H}^{-1/2}_{\perp} ( {\rm curl}_{\Gamma}, \Gamma)    \quad \mbox{and}\quad {\bf j} \in {\bf H}^{-1/2}_{||} ({\rm div}_{\Gamma}, \Gamma ).
\]
For $x\in\Omega_+$, we propose an ansatz of the form
\begin{equation}
\label{eq:ansatz}
{\bf E}^{scat}(x) = {\bf curl} \int_\Gamma G\left(x,y;s/c_0\right)\gamma_t{\bf m}(y)dy - \frac{1}{s\epsilon_0}\,{\bf curl\,curl}\int_\Gamma G\left(x,y;s/c_0\right){\bf j}(y)dy,
\end{equation}
where   
\[
G(x,y;t) : =  \frac{1} { 4 \pi |x-y|}\, e^{-t|x-y|} 
\]
 is the  fundamental solution of the Yukawa potential  equation in $\mathbb{R}^3$, and
 $c_0 := (\mu_0 \, \varepsilon_0)^{-1/2} $ is  the speed  of light in the vacuum \cite{So:1964}. As we shall now demonstrate, the function defined by \eqref{eq:ansatz} is indeed a solution of \eqref{eq:h2.8b}. From the representation above, we have
 {\small \begin{align}
 \nonumber
 {\bf curl\,E}^{scat}(x)\! =\,& {\bf curl\,curl}\!\! \int_\Gamma\! G\left(x,y;s/c_0\right)\gamma_t{\bf m}(y)dy - \frac{1}{s\epsilon_0}\,{\bf curl\,curl\,curl}\!\!\int_\Gamma G\!\left(x,y;s/c_0\right){\bf j}(y)dy\\
 \nonumber
=\,& {\bf curl\,curl} \int_\Gamma G\left(x,y;s/c_0\right)\gamma_t{\bf m}(y)dy + \frac{1}{s\epsilon_0}\,\frac{s^2}{c_0^2}{\bf curl}\int_\Gamma G\left(x,y;s/c_0\right){\bf j}(y)dy\\
\label{eq:IntRepCurl}
=\,& {\bf curl\,curl} \int_\Gamma G\left(x,y;s/c_0\right)\gamma_t{\bf m}(y)dy + s\mu_0{\bf curl}\int_\Gamma G\left(x,y;s/c_0\right){\bf j}(y)dy,
 \end{align} }
and thus
{\small \[
 {\bf curl\,curl\,E}^{scat}(x)\! = -\frac{s^2}{c_0^2}{\bf curl}\!\! \int_\Gamma\! G\left(x,y;s/c_0\right)\gamma_t{\bf m}(y)dy + s\mu_0{\bf curl\, curl}\!\!\int_\Gamma\! G\left(x,y;s/c_0\right){\bf j}(y)dy.
\]}
 Multiplying \eqref{eq:ansatz} by $(s/c_0)^2$ and adding the result to the expression above it follows that
 \[
 {\bf curl \, curl E}^{scat} + (s/c_0)^2{\bf E}^{scat} = \mathbf 0 \qquad \text{ in } \Omega_+.
 \]
Making use of the electromagnetic layer potentials, the integral representations \eqref{eq:ansatz} and \eqref{eq:IntRepCurl} for the scattered electric field and its curl can be written succintly as
\[
\mathbf E^{scat} = \widetilde{\mathcal D}(s/c_0)\mathbf m - \frac{1}{s\epsilon_0}\mathcal S(s/c_0)\mathbf j \quad \text{ and } \quad {\bf curl\, E}^{scat} = \widetilde{\mathcal S}(s/c_0)\mathbf m + s\mu_0\mathcal D(s/c_0)\mathbf j.
\]
If we then extend ${\bf E}^{scat}$ by zero in $\Omega_-$, it is possible to use the jump properties of the electromagnetic layer potentials and the boundary integral operators introduced in the previous section to show that the traces of an electric field ${\bf E}^{scat}$ defined using the integral representation above satisfy
\begin{equation}
\label{eq:traceIDs}
\pi_t^+{\bf E}^{scat} = {\bf m} \in \mathbf H^{-1/2}_{\perp}({\bf curl}_\Gamma,\Gamma), \qquad \gamma_t^+\,{\bf curl E}^{scat} = s\mu_0{\bf j} \in \mathbf H^{-1/2}_{\|}({\rm div}_\Gamma,\Gamma), 
\end{equation}
and the system of boundary integral equations
\begin{equation}
\label{eq:BIE}
\mathcal C_{\Omega_+}(s/c_0)\!\begin{pmatrix}
{\bf m} \\[1.2ex] {\bf j}
\end{pmatrix} :=\!
\begin{pmatrix}
1/2 + \widetilde{\mathcal K}(s/c_0) & -(s\epsilon_0)^{-1}\mathcal V(s/c_0) \\[1.2ex]
(s\mu_0)^{-1}\widetilde{\mathcal V}(s/c_0) & 1/2 + \mathcal K(s/c_0)
\end{pmatrix}\!
\begin{pmatrix}
{\bf m} \\[1.2ex] {\bf j}
\end{pmatrix} \!=\!
\begin{pmatrix}
\pi_t^+ {\bf E}^{scat} \\[1.2ex]
(s\mu_0)^{-1}\gamma_t^+{\bf curl E}^{scat}\!
\end{pmatrix}\!.
\end{equation}
Substituting the right hand side of the expression above by appropriate boundary or transmission conditions leads to the BIE for the exterior scattered field in terms of the traces of the incident wave and the transmitted field. The operator $ \mathcal{C}_{\Omega_+}(s) $ defineds implicitly in the equation above is known as the {\it Calder\'{o}n projector }for $\Omega_+$---see \cite{HsWe:2021} and \cite[p.63] {Ce:1996}.

We will end this section by proving that \eqref{eq:ansatz} is indeed the integral representation of the unique distributional solution of a more general Laplace-domain transmission problem for the Maxwell system.

 %=======================
\begin{proposition} \label{pro:h3.1} 
%=========================
Given $ ({\bf m, j} ) \in {\bf H}_{\perp}^{-1/2}({\rm curl}_{\Gamma}, \Gamma) \times  {\bf H}_{||}^{-1/2} ({\rm div}_{\Gamma} , \Gamma) $, the function ${\bf E}^{scat} \in  {\bf H}({ \bf curl},   \mathbb{R}^3\setminus \Gamma) \cap  {\bf H}({ \bf curl}^2,  \mathbb{R}^3\setminus \Gamma ) $ represented for  $x\in\mathbb{R}^3 \setminus \Gamma$  by
 \begin{equation}
 \label{eq:h3.16}
\mathbf E^{scat} = \widetilde{\mathcal D}(s/c_0)\mathbf m - \frac{1}{s\epsilon_0}\mathcal S(s/c_0)\mathbf j
\end{equation}
is the unique distributional solution of the  transmission problem 
\begin{subequations}
\label{eq:h3.17all}
\begin{alignat}{6}
 \label{eq:h3.17}
{ \bf curl \, curl \, E}^{scat} +  \frac{s^2}{c_0^2} {\bf E}^{scat} =\,& {\bf 0}  && \quad \mbox{ in } \mathbb{R}^3 \setminus \Gamma, \\
\label{eq:h3.18a}
 \jump{\pi_t {\bf E}^{scat} } =\,& -{\bf m}  &&\quad \text{ on } \Gamma,\\
 \label{eq:h3.18b}
\jump{\gamma_t {\bf curl\, E}^{scat}} =\,& -s\mu_0{\bf j}  &&\quad \text{ on } \Gamma.
\end{alignat}
\end{subequations}
Moreover, the solution $ {\bf E}^{scat}$ satisfies  the stability estimate  
\begin{equation}
\label{eq:h3.19}
 \| {\bf E}^{scat}\|_{{\bf curl}, \mathbb{R}^3 \setminus \Gamma} \lesssim \frac{|s|^2} {\sigma \underline{\sigma}^2} \left( \| {\bf  m}  \|_{-1/2,\,  {\rm curl}_{\Gamma},\,  \Gamma}  + \|  { \bf j}  \|_{-1/2,\,  {\rm div}_{\Gamma},  \Gamma} \right). 
\end{equation}
\end{proposition} 
%===============
\begin{proof}
%===============
We start by extending ${\bf m} $ to a function $ \widetilde{{\bf E}} \in {\bf H}({\bf curl}, \mathbb{R}^3 \setminus \Gamma)$ such that
$\jump{\pi_t \widetilde{{\bf E}} }\!\!= - {\bf m}$. This can be done easily by taking a lifting of ${\bf m}$ into $\Omega_+$ and extending it by zero for $x\in\Omega_-$. Now, for $\mathbf E^{scat}$ satisfying \eqref{eq:h3.16}, we define ${\bf E}_0 := {\bf E}^{scat} - \widetilde{\bf E}$ and observe that $\jump{\pi_t ( {\bf E}^{scat}  - \widetilde{ {\bf E} })} = {\bf 0}$. 

Multiplying \eqref{eq:h3.17} by a test function $ {\bf F} \in \left\{{\bf H}({\bf curl},  \mathbb{R}^3\setminus \Gamma): \jump{\pi_t {\bf F}}=\boldsymbol 0\right\}$ and integrating by parts, we obtain the  variational formulation for \eqref{eq:h3.17all}:
\[
 a\left({\bf E}_0, {\bf F} ; s/c_0\right)_{\mathbb{R}^3 \setminus \Gamma} = - a( \widetilde{\bf E}, {\bf F}; s/c_0 )_{\mathbb{R}^3 \setminus \Gamma} + s\mu_0\langle {\bf j}, \pi_t {\bf F}\rangle_{\Gamma},
\]
where the bilinear form $a(\cdot,\cdot)_{\mathbb{R}^3 \setminus \Gamma}$ is the one defined in \eqref{eq:h3.2}. Using the associated energy norm \eqref{eq:energynorm} and the properties \eqref{eq:energynormP1}, \eqref{eq:energynormP2}, and \eqref{eq:energynormP3} it is easy to see that  
\[
\sigma \underline{\sigma}^2 
\triple{ {\bf E}_0}^2_{|1|} \leq {\rm Re }\left\{\bar{s} \;  a\left( {\bf E}_0, \overline{\bf E}_0; s/c_0\right)_{\mathbb{R}^3 \setminus \Gamma}  \right\} 
\]
and hence, by the Lax-Milgram lemma, there exists a unique solution ${\bf E}_0 \in {\bf H} ({\bf curl}, \mathbb{R}^3 \setminus \Gamma)$ for the variational problem above. Consequently
\[
{\bf E}^{scat} :=  {\bf E}_0 +  \widetilde{\bf E} \in {\bf H}( {\bf curl},  \mathbb{R}^3 \setminus \Gamma)
\]
is the unique solution of the transmission problem defined by  \eqref{eq:h3.17all}. 

Now, for the stability estimate \eqref{eq:h3.19}, we begin with 
\begin{align*}
\sigma \triple{ {\bf E}_0}^2_{|s/c_0|} 
& = {\rm Re} \left\{\bar{s} \; a({\bf E}_0, \overline{\bf E}_0; s/c_0)_{\mathbb{R}^3 \setminus \Gamma}  \right\} \\
& = {\rm Re } \left\{ \bar{s}s\mu_0 \langle {\bf j}, \pi_t\, \overline{\bf E}_0\rangle_{\Gamma} -  \bar{s}\, a( \widetilde{\bf E},  \overline{\bf E}_0; s/c_0 )_{\mathbb{R}^3 \setminus \Gamma} \right\}\\
&\lesssim |s|^2 \, \| {\bf j}\|_{-1/2, {\rm div}_{\Gamma}, \Gamma} \; \triple {{\bf E}_0}_{|1|} + |s| \, 
\triple {\widetilde{\bf E}}_{|s/c_0|} \, \triple{{\bf E}_0}_{|s/c_0|}\\
&\lesssim \frac{|s|^2}{\underline{\sigma}}\| {\bf j }\|_{-1/2, {\rm div}_{\Gamma}, \Gamma}  \, \triple {{\bf E}_0}_{|s/c_0|} +  \frac{|s|^2}{\underline{\sigma}} \|\widetilde{\bf E}\|_{{\bf curl}, \mathbb{R}^3\setminus{\Gamma}} \triple{ {\bf E}_0}_{|s/c_0|} \\ 
& \lesssim \frac{|s|^2}{\underline{\sigma}}\left( \|{\bf j}\|_{-1/2, {\rm div}_{\Gamma}, \Gamma}
+ \| \widetilde{ \bf E } \| _{{\bf curl}, \mathbb{R}^3\setminus{\Gamma}} \right) \triple {{\bf E}_0}_{|s/c_0|}, 
\end{align*}
which together with \eqref{eq:energynormP1} implies 
\[
 \| {\bf E}_0\|_{{\bf curl,}  \mathbb{R}^3\setminus \Gamma} 
\lesssim \frac{|s|^2}{\sigma\underline{\sigma}^2}\left(\|{\bf j}\|_{-1/2, {\rm div}_{\Gamma}, \, \Gamma}
+ \|\widetilde{ \bf E}\| _{{\bf curl}, \mathbb{R}^3\setminus{\Gamma}}   \right). 
\]
Hence, from the definition of ${\bf E}$ and of the norm $\|\cdot\|_{-1/2, {\rm curl}_{\Gamma}, \Gamma}$ it follows that
\begin{align*}
\| {\bf E}^{scat}\|_{{\bf curl,}  \mathbb{R}^3 \setminus \Gamma}  
 \leq &  \left(\|{\bf E}_0\|_{{\bf curl,}  \mathbb{R}^3\setminus \Gamma } + \|  \widetilde{\bf E}\|_{{\bf curl,}  \mathbb{R}^3 \setminus \Gamma} \right)\\
 \lesssim  & \left( \frac{|s|^2}{\sigma\, \underline{\sigma}^2}\|{\bf j }\|_{-1/2, {\rm div}_{\Gamma}, \Gamma}+
 \left(\frac{|s|^2}{\sigma\, \underline{\sigma}^2} +1 \right) \|  \widetilde{ \bf E }\|_{{\bf curl,}  \mathbb{R}^3 \setminus \Gamma}  \right)\\
  \lesssim  & \left( \frac{|s|^2}{\sigma\, \underline{\sigma}^2} \|{\bf j }\|_{-1/2, {\rm div}_{\Gamma}, \Gamma}+
 2\frac{|s|^2}{\sigma\, \underline{\sigma}^2}\|  \widetilde{ \bf E}\|_{{\bf curl,}  \mathbb{R}^3 \setminus \Gamma}  \right)\\
\lesssim  &  \frac{|s|^2}{\sigma\, \underline{\sigma}^2}   \left(\|{\bf j}\|_{-1/2, {\rm div}_{\Gamma}, \Gamma}
+ \| {\bf m} \|_{-1/2, {\rm curl}_{\Gamma}, \Gamma}  \right).
\end{align*}
This completes the proof of the proposition. 
\end{proof}
%
%
%
 %==========================
\section{Boundary-field formulation}\label{sec:5}
%==========================
%
As we anticipated in the previous section, when using the electromagnetic layer potential representation \eqref{eq:ansatz} to recast the problem as a boundary-field formulation the transmission conditions \eqref{eq:h2.8c} and \eqref{eq:h2.8d} become key tools in obtaining a system of BIEs for the scattered field. In view of the identities \eqref{eq:traceIDs} and the properties of the tangential traces, these conditions become: 
 \begin{alignat*}{6}
\gamma_t^- {\bf curl \, E} =\,& s\mu{\bf j} +  (\mu/\mu_0)\gamma_t ^+{\bf curl\,E}^{inc} \qquad&& \text{ on } \Gamma,\\
 \pi^-_t {\bf E} =\,& { \bf m }  + \pi^+_t { \bf E}^{inc}   \qquad&& \text{ on } \Gamma. 
\end{alignat*}
Recalling that the electric field in the bounded region $\Omega_-$ satisfies the variational problem \eqref{eq:h3.1} and its operator counterpart \eqref{eq:h3.3}, the first transmission condition simply implies that if $ {\bf j} $ is known we only need to solve an interior Neumann problem for $ {\bf E}$ in $\Omega_-$. This provides enough information to solve the interior field problem and couples it to the solution of the system of BIEs. The second transmission condition above can be substituted in the right hand side of the first row of \eqref{eq:BIE}, further coupling the interior and exterior problems. However, a \textit{third} condition is missing to close the problem. The key observation is that, since the external scattered field will be extended by zero in $\Omega_-$, we must then have
\[
\gamma_t^{-}{\bf curl\, E}^{scat}={\bf 0}.
\]
Using the electromagnetic layer potential representation of ${\bf curl\,E}^{scat}$ and its jump properties, this condition can be written in terms of boundary integral operators as
\[
(s\mu_0)^{-1}\gamma_t^{-}{\bf curl\,E}^{scat} = (s\mu_0) \left(-1/2+\mathcal K(s/c_0)\right){\bf j} +\widetilde{\mathcal V}(s/c_0){\bf m} = \boldsymbol 0.
\]
Putting all of these ingredients together the problem can be restated as that of, given data
\[
( \pi^+_t{\bf E}^{inc},  \gamma_t ^+{\bf curl \,E}^{inc}, {\bf J}  ) \in {\bf H}^{-1/2}_{\perp} ( {\rm curl}_{\Gamma}, \Gamma) \times {\bf H}^{-1/2}_{\|} ( {\rm div}_{\Gamma}, \Gamma) \times \mathbf L^2(\Omega_-),
\] 
finding
\[
({\bf E}, {\bf j}, {\bf  m})\in {\bf H} ({\bf curl}, \Omega_-)\times {\bf H}^{-1/2}_{\|}( {\rm div}_{\Gamma}, \Gamma)\times {\bf H}_{\perp}^{-1/2}( {\rm curl}_{\Gamma}, \Gamma)
\]
satisfying the boundary-field system:
{\small \begin{equation}
\label{eq:h4.5}
\!\!\!\begin{pmatrix}
\!{\bf A}_{\Omega_-}(s/c) & s\mu\left(\pi_t^-\right)^\prime & 0 \\
\pi_t^- & \;(s\epsilon_0)^{-1}\mathcal V(s/c_0)\; & -1/2 - \widetilde{\mathcal K}(s/c_0) \!\!\! \\
0 & \left(-1/2+\mathcal K(s/c_0)\right) & (s\mu_0)^{-1}\widetilde{\mathcal V}(s/c_0) 
\end{pmatrix}\!\!\!
\begin{pmatrix}
{\bf E} \\ {\bf j} \\ {\bf m}
\end{pmatrix}
\!\! = \!\!
\begin{pmatrix}
\!-s\mu {\bf J} \!-(\mu/\mu_0)\left(\pi_t^-\right)^\prime\!\!\gamma^+_t {\bf curl\, E}^{inc} \\ \pi^+_t {\bf E}^{inc} \\ {\bf 0}
\end{pmatrix}\!.
\end{equation}}
We will show that this problem is indeed solvable in the following section.
%
% ==============================
\subsection{An equivalent BVP}
% ==============================
%
Our next task will be to show that Equation \eqref{eq:h4.5} has a unique solution. In \cite[Lemma 3.5]{NiKoLu:2022}, Nick, Kov\'{a}cs \& Lubich established this result by recasting the problem as a first order system and analyzing the modified Stratton-Chu formula. However, in order to obtain sharper estimates, in this work we decide to analyze the problem indirectly. Following Laliena and Sayas \cite{LaSa:2009} in acoustics, and Hsiao, S\'anchez-Vizuet \& Wendland \cite{HsSaWe:2022} in elastodynamics, we will formulate a non standard boundary value problem for ${\bf E}$ in $\Omega_-$ and ${\bf E}^{scat} \in \mathbb{R}^3 \setminus \Gamma$ and will show its well-posedness. We will also show that a solution  triplet $\left( {\bf E}, -(s\mu_0)^{-1}\jump{\gamma_t{\bf curl \, E}^{scat} }, -\jump{\pi_t{\bf E}^{scat}}\right)$ of said boundary value problem can be used to establish the solution of the nonlocal boundary problem \eqref{eq:h4.5} and that the reciprocal is also true. This will justify the notation for the variables of the auxiliary problem and will show the well posedness of \eqref{eq:h4.5}.

\begin{proposition} 
Suppose that $( {\bf E}, {\bf j}, \bf{ m}) $ is a solution of \eqref{eq:h4.5} and let us define for $x\in\mathbb{R}^3\setminus\Gamma$
\begin{equation}
\label{eq:h4.7}
{\bf E}^{scat}\! (x) \!\!:= \widetilde{\mathcal D}(s/c_0){\bf m} - \frac{1}{s\epsilon_0}\mathcal S(s/c_0){\bf j}
\end{equation}
Then $({\bf E},{\bf E}^{scat})\in  {\bf H}({\bf curl},\Omega_-)\times {\bf H}({\bf curl},\mathbb R^3\setminus\Gamma)$ is a solution of the transmission problem
{\small\begin{subequations}
\label{eq:nonstdtp}
\begin{alignat}{6}
\label{eq:nonstdtpA}
{\bf A}_{\Omega_-}\!(\!s/c\!){\bf E} + \! (\mu/\!\mu_0\!)\!\!\left(\pi^-\right)^\prime\!\!\jump{\gamma_t{\bf curl\,E}^{scat}}\!\!=\! & -\!s\mu {\bf J} \!-\!(\mu/\!\mu_0\!)\!\!\left(\pi_t^-\right)^\prime\!\!\gamma^+_t \!{\bf curl\, E}^{inc} &\,& \text{ in }{\bf H}^\prime({\bf curl},\Omega_-\!),\\
\label{eq:nonstdtpB}
{ \bf curl \, curl \, E}^{scat} +  \frac{s^2}{c_0^2} {\bf E}^{scat} =\,& {\bf 0}  & \, & \text{ in } \mathbb{R}^3 \setminus \Gamma,\\
\label{eq:nonstdtpC}
\pi^-_t {\bf E} - \pi^+_t { \bf E}^{scat}=\, & \pi^+_t { \bf E}^{inc}  &\,& \text{ in } {\bf H}^{-1/2}_{\perp}\! ({\rm curl}_{\Gamma}, \Gamma),\\
 \label{eq:nonstdtpD}
 \gamma_t^{-}{\bf curl\, E}^{scat}=\,&{\bf 0} &\,& \text{ in } {\bf H}^{-1/2}_{\|}\! ( {\rm div}_{\Gamma}, \Gamma).
\end{alignat}
\end{subequations}}
Conversely, given $({\bf E},{\bf E}^{scat})$ satisfying \eqref{eq:nonstdtp}, the triplet
\[
( {\bf E}, {\bf j}, {\bf m}) := \left( {\bf E}, -(s\mu_0)^{-1}\jump{\gamma_t{\bf curl \, E}^{scat} }, -\jump{\pi_t{\bf E}^{scat}}\right)
\]
is a solution to the boundary-field system \eqref{eq:h4.5}.
\end{proposition}
\begin{proof}
From the definition of ${\bf E}^{scat}$ and Proposition \ref{pro:h3.1}, it follows that ${\bf E}^{scat}$ satisfies \eqref{eq:nonstdtpB} and
\[
{\bf j} = -(s\mu_0)^{-1}\jump{\gamma_t{\bf curl\, E}^{scat}} \quad \text{ and } \quad {\bf m} = - \jump{\pi_t{\bf E}^{scat}}.
\]
Hence, using these expressions, the trace identities \eqref{eq:LPtraces}, and the definition of the boundary integral operators on the integral representation of ${\bf E}^{scat}$, it follows that conditions \eqref{eq:nonstdtpC} and \eqref{eq:nonstdtpD} are restatements of the second and third equations in \eqref{eq:h4.5}.

Conversely, \eqref{eq:nonstdtpB} and the definition of ${\bf j}$ and ${\bf m}$ imply the integral representation \eqref{eq:h4.7}. This representation enables the use of the identities \eqref{eq:LPtraces} and the properties of the boundary integral operators from which a simple algrebraic calculation shows the equivalence between \eqref{eq:nonstdtpA}, \eqref{eq:nonstdtpC}, and \eqref{eq:nonstdtpD} to the system  \eqref{eq:h4.5}.  
\end{proof}

As we shall now prove, the transmission problem \eqref{eq:nonstdtp} above can be shown to be well-posed through variational means. Thus, the relevance of the previous result is that it extends the well posedness to the boundary field formulation \eqref{eq:h4.5}. 

\begin{proposition}
\label{prop:EquivalentBVP}
Consider the function spaces
\[
\mathbf H^* := \left\{{\bf E}\in {\bf H}({\bf curl},\mathbb R^3\setminus\Gamma): \gamma_t^-{\bf curl\, E} = \bf 0 \right\} \qquad \text{ and } \qquad
\mathbb H : = {\bf H}({\bf curl},\Omega_-)\times \mathbf H^*.
\]
The transmission problem \eqref{eq:nonstdtp} is equivalent to that of finding $({\bf E},{\bf E}^{scat})\in \mathbb H $ satisfying
\begin{subequations}
\label{eq:variationalprob}
\begin{alignat}{6}
\label{eq:variationaltrace}
\pi^-_t {\bf E} - \pi^+_t { \bf E}^{scat}=\,& \pi^+_t { \bf E}^{inc}   &\quad& \text{ in } {\bf H}^{-1/2}_{\perp} ( {\rm curl}_{\Gamma}, \Gamma),\\[1.5ex]
 \label{eq:variationaleq}
{\rm B}\left(({\bf E},{\bf E}^{scat}),({\bf F},{\bf G})\right) =\,& \ell\left(({\bf F},{\bf G})\right) &\quad& \; \forall \; ({\bf F},{\bf G})\in \mathbb H,
\end{alignat}
where
\begin{align}
\label{eq:variationalprobBilinear}
{\rm B}\left(({\bf E},{\bf E}^{scat}),({\bf F},{\bf G})\right) :=&\quad (\mu_0/\mu)a({\bf E},{\bf F};s/c)_{\Omega_-} + \, a({\bf E}^{scat},{\bf G};s/c_0)_{\mathbb R^3\setminus\Gamma} \\
\nonumber
& + \langle\gamma^+_t{\bf curl\,E}^{scat}\!,\pi_t^-{\bf F}- \pi_t^+{\bf G}\rangle_\Gamma ,\\[1.5ex]
\label{eq:variationalprobLinear}
\ell\left(({\bf F},{\bf G})\right) :=\,& -s\mu_0\left( {\bf J},{\bf F }\right)_{\Omega_-} -\langle\gamma^+_t {\bf curl\, E}^{inc}\!,\pi_t^-{\bf F}\rangle_\Gamma.
\end{align}
\end{subequations}
\end{proposition}
\begin{proof}
Let $({\bf E},{\bf E}^{scat})$ be a solution to \eqref{eq:nonstdtp} and note that \eqref{eq:variationaltrace} is satisfied automatically. We first multiply  \eqref{eq:nonstdtpA} by $(\mu_0/\mu)$ and test with ${\bf F}\in {\bf H}({\bf curl},\Omega_-)$ to obtain
\begin{equation}
\label{eq:variationalpt1}
 (\mu_0/\mu)a({\bf E},{\bf F};s/c)_{\Omega_-}  + \langle\gamma^+_t{\bf curl\,E}^{scat}\!,\pi^-{\bf F}\rangle_\Gamma = -s\mu_0\left( {\bf J},{\bf F }\right)_{\Omega_-} -\langle\gamma^+_t {\bf curl\, E}^{inc}\!,\pi_t^-{\bf F}\rangle_\Gamma,
\end{equation}
in view of \eqref{eq:nonstdtpD}. Then, testing \eqref{eq:nonstdtpB} with ${\bf G}\in{\bf H}^*$, splitting the domain of integration into $\Omega_-$ and $\Omega_+$ and using the integration by parts formulas \eqref{eq:integrationbyparts} yields
{\small \begin{align}
\nonumber
\left({\bf curl\,curl\,E}^{scat}\!,{\bf G}\right)_{\mathbb R^3\setminus\Gamma} \!\!+\! (s/c_0)^2\!\left({\bf E}^{scat}\!,{\bf G}\right)_{\mathbb R^3\setminus\Gamma}
\! =& \;\; \left({\bf curl\,E}^{scat}\!,{\bf curl\,G}\right)_{\mathbb R^3\setminus\Gamma} \!\!+ \!(s/c_0)^2\!\!\left({\bf E}^{scat}\!,{\bf G}\right)_{\mathbb R^3\setminus\Gamma} \\[1.5ex]
\nonumber
&\!-\!\langle\gamma^+_t{\bf curl\,E}^{scat}\!,\pi^+{\bf G}\rangle_\Gamma \\[1.5ex]
\nonumber
 = &\;  a({\bf E}^{scat},{\bf G};s/c_0)_{\mathbb R^3\setminus\Gamma}  - \langle\gamma^+_t{\bf curl\,E}^{scat}\!,\pi^+{\bf G}\rangle_\Gamma\\[1.5ex]
 \label{eq:variationalpt2}
 = &\; 0.
\end{align}}
where we once again made use of \eqref{eq:nonstdtpD}. Adding the two expressions \eqref{eq:variationalpt1} and \eqref{eq:variationalpt2} we recover \eqref{eq:variationaleq}.

To prove the converse, we let $({\bf E},{\bf E}^{scat})$ solve the variational problem \eqref{eq:variationalprob} and observe that \eqref{eq:nonstdtpC} and \eqref{eq:nonstdtpD} are satisfied. We then consider tests of the form $({\bf F},{\bf 0})$ and $({\bf 0},{\bf G})$ in \eqref{eq:variationaleq}, which shows that both
{\small \begin{align}
\label{eq:var1}
(\mu_0/\mu)a({\bf E},{\bf F};s/c)_{\Omega_-} + \langle\gamma^+_t{\bf curl\,E}^{scat}\!, \pi^-{\bf F}\rangle_\Gamma =\,&  -s\mu_0\left( {\bf J},{\bf F }\right)_{\Omega_-} -\langle\gamma^+_t {\bf curl\, E}^{inc}\!,\pi_t^-{\bf F}\rangle_\Gamma,\\
\label{eq:var2}
a({\bf E}^{scat},{\bf G};s/c_0)_{\mathbb R^3\setminus\Gamma} - \langle\gamma^+_t{\bf curl\,E}^{scat}\!,\pi^+{\bf G} \rangle_\Gamma =\,& 0,
\end{align}}
must hold independently of each other. Since $\gamma^-_t{\bf E}^{scat} = {\bf 0}$, equation \eqref{eq:var1} is immediately equivalent to \eqref{eq:nonstdtpA}.  Moreover, if in \eqref{eq:var2} the test function ${\bf G}$ belongs to $C^\infty(\mathbb R^3\setminus\Gamma)$ and is supported in a compact set that does not intersect $\Gamma$, we obtain
\[
a({\bf E}^{scat},{\bf G};s/c_0)_{\mathbb R^3\setminus\Gamma} = 0
\]
which is simply the distributional form of \eqref{eq:nonstdtpB}, completing the proof.
\end{proof}

%
% ==============================================================
\section{Existence, uniqueness and stability}\label{sec:6}
% ==============================================================
%

Before proceeding to show the well posedness of the problem, we remark that in the product space $\mathbb H$ defined in the Proposition \ref{prop:EquivalentBVP}, the operator $\pi({\bf E},{\bf F}):= \pi_t^-\mathbf E - \pi_t^+{\bf F}$ is in effect an abstract trace. Due to its definition in terms of the inner and outer traces $\pi_t^-$ and $\pi_t^+$, it is possible to define a linear bounded pseudoinverse that we will denote by $\pi^\dagger$. In this framework, condition \eqref{eq:variationaltrace} is in fact an essential Dirichlet boundary condition. Therefore, as it is customary for problems with non-homogeneous essential boundary conditions, we will now define the kernel of the trace by
\[
\mathbb H_0 := \left\{({\bf E},{\bf F})\in \mathbb H : \pi(\mathbf E,{\bf F}) = 0\right\},
\] 
and will procceed to show well posedness of a general Dirichlet boundary value problem involving the operator ${\rm B}$ restricted to $\mathbb H_0$. This result will be extended to the non-homogeneous case by a standard argument and will finally be applied to \eqref{eq:variationalprob}.

In the sequel, to keep the notation as compact as possible, we will write ${\bf X},{\bf Y}\in \mathbb H$ to denote pairs of the form
\[
{\bf X} = ({\bf E},{\bf F}) \quad \text{ and } \quad {\bf Y} = ({\bf U},{\bf V}) \;\in \mathbb H.
\]
Similarly, for the product norm we will write
\[
\triple{\bf X}_{|s|}^2 = \triple{\bf E}_{|s|,\mathbb R^3\setminus\Gamma}^2 + \triple{\bf F}_{|s|,\Omega_-}^2.
\] 
\begin{proposition}
\label{prop:6.1}
Let $\mathbb H^\prime$ denote the dual space of $\mathbb H$ and $\pi\mathbb H$ be its trace space, and ${\rm B}$ the bilinear form defined in \eqref{eq:variationalprobBilinear}. Given $f\in \mathbb H^\prime$ and ${\bf d}\in \pi\mathbb H $, the problem of finding ${\bf X}=({\bf E},{\bf F})\in\mathbb H$ satisfying
\begin{align*}
\pi{\bf X} =\,& {\bf d}, \\
{\rm B}\left({\bf X},{\bf Y}\right) =\,& f\left({\bf Y}\right) \qquad \forall\; {\bf Y}=({\bf U},{\bf V})\in \mathbb H_0, 
\end{align*}
has a unique solution satisfying the stability estimate
\begin{equation}
\label{eq:stabilityEstimate1}
\triple{{\bf X}}_{|s|} \lesssim \frac{|s|^2}{\sigma\underline\sigma^2}\left( \|f\| + 2\|\pi^\dagger\|\|{\bf d}\|_{-1/2, \mathrm{{\rm curl}_{\Gamma}}, \Gamma}\right).
\end{equation}
\end{proposition}

\begin{proof}
It is easy to verify that the bilinear form ${\rm B}$ has a bound of the form
\[
|{\rm B}({\bf X}, {\bf Y})|\lesssim\triple{\bf X}_{|s|}\triple{\bf Y}_{|s|}.
\]
Let ${\bf X}_0:=({\bf E_0},{\bf F_0})= ({\bf E},{\bf F})- \pi^\dagger{\bf d}\in \mathbb H_0$ and notice that, in this space, ${\rm B}$ is strongly elliptic in the sense that
\[
 \sigma\triple{{\bf X}_0}_{|s|}^2 \lesssim {\rm Re} \left(\overline s \,{\rm B}\left({\bf X}_0,\overline{{\bf X}}_0\right)\right),
\]
where the hidden constant depends only on the physical parameters $\epsilon,\epsilon_0, \mu,\mu_0$ and on the geometry of $\Omega_-$. Hence, by the Lax-Milgram lemma, the variational problem is uniquely solvable in $\mathbb H_0$ and we can then recover ${\bf X} = {\bf X}_0 + \pi^\dagger{\bf d}$.

To determine the stability constants we use the equivalence relations \eqref{eq:energynormP1} and compute
\begin{align*}
\left|{\rm Re} \left(\overline s \,{\rm B}\left({\bf X}_0,\overline{{\bf X}}_0\right)\right)\right| =\,&\left|{\rm Re}\left(\overline s\left(f(\overline{\bf X}_0) - {\rm B}\left(\pi^\dagger{\bf d},\overline{{\bf X}}_0\right)\right)\right)\right| \\
\lesssim \,& |s|\left( \|f\|\triple{{\bf X}_0}_{|1|} + \triple{\pi^\dagger{\bf d}}_{|s|}\triple{{\bf X}_0}_{|s|}\right) \\
\lesssim \,& |s|\triple{{\bf X}_0}_{|s|}\left( \frac{\|f\|}{\underline\sigma} + \triple{\pi^\dagger{\bf d}}_{|s|}\right)\\
\lesssim \,& |s|\triple{{\bf X}_0}_{|s|}\left( \frac{\|f\|}{\underline\sigma} + \frac{|s|}{\underline\sigma}\|\pi^\dagger\|\|{\bf d}\|_{-1/2, \mathrm{{\rm curl}_{\Gamma}}, \Gamma}\right)\\
\lesssim \,& \frac{|s|^2}{\underline\sigma^2}\triple{{\bf X}_0}_{|s|}\left( \|f\| + \|\pi^\dagger\|\|{\bf d}\|_{-1/2, \mathrm{{\rm curl}_{\Gamma}}, \Gamma}\right).
\end{align*}
Combining this result with the ellipticity estimate obtained above yields
\begin{equation}
\label{eq:stab1}
\triple{{\bf X}_0}_{|s|} \lesssim \frac{|s|^2}{\sigma\underline\sigma^2}\left( \|f\| + \|\pi^\dagger\|\|{\bf d}\|_{-1/2, \mathrm{{\rm curl}_{\Gamma}}, \Gamma}\right).
\end{equation}
On the other hand, recalling that ${\bf X} = {\bf X}_0 + \pi^\dagger{\bf d}$ we obtain:
{\small\[
\triple{{\bf X}}_{|s|}\leq \triple{{\bf X}_0}_{|s|} + \triple{\pi^\dagger{\bf d}}_{|s|} \leq \triple{{\bf X}_0}_{|s|} + \frac{|s|}{\underline\sigma}\|\pi^\dagger\|\|{\bf d}\|_{-1/2, \mathrm{{\rm curl}_{\Gamma}}, \Gamma} \leq \triple{{\bf X}_0}_{|s|} + \frac{|s|^2}{\sigma\underline\sigma^2}\|\pi^\dagger\|\|{\bf d}\|_{-1/2, \mathrm{{\rm curl}_{\Gamma}}, \Gamma}.
\]}
This estimate together with \eqref{eq:stab1} leads to \eqref{eq:stabilityEstimate1}.
\end{proof}
We are finally in position to prove the well posedness of the boundary field problem \eqref{eq:h4.5}. Thanks to the machinery developed in this and the previous section, the proof of the following Theorem will seem simple
\begin{theorem}
\label{thm:6.1}
Given data
\[
( {\bf J}, \gamma_t ^+{\bf curl \,E}^{inc},\pi^+_t{\bf E}^{inc},    ) \in \mathbf L^2(\Omega_-)\times {\bf H}^{-1/2}_{\|} ( {\rm div}_{\Gamma}, \Gamma) \times {\bf H}^{-1/2}_{\perp} ( {\rm curl}_{\Gamma}, \Gamma),
\] 
the boundary-field problem of finding functions
\[
({\bf E}, {\bf j}, {\bf  m})\in {\bf H} ({\bf curl}, \Omega_-)\times {\bf H}^{-1/2}_{\|}( {\rm div}_{\Gamma}, \Gamma)\times {\bf H}_{\perp}^{-1/2}( {\rm curl}_{\Gamma}, \Gamma)
\]
satisfying
{\small \[
\!\!\!\begin{pmatrix}
\!{\bf A}_{\Omega_-}(s/c) & s\mu\left(\pi_t^-\right)^\prime & 0 \\
\pi_t^- & \;(s\epsilon_0)^{-1}\mathcal V(s/c_0)\; & -1/2 - \widetilde{\mathcal K}(s/c_0) \!\!\! \\
0 & \left(-1/2+\mathcal K(s/c_0)\right) & (s\mu_0)^{-1}\widetilde{\mathcal V}(s/c_0) 
\end{pmatrix}\!\!\!
\begin{pmatrix}
{\bf E} \\ {\bf j} \\ {\bf m}
\end{pmatrix}
\!\! = \!\!
\begin{pmatrix}
\!-s\mu {\bf J} \!-(\mu/\mu_0)\left(\pi_t^-\right)^\prime\!\!\gamma^+_t {\bf curl\, E}^{inc} \\ \pi^+_t {\bf E}^{inc} \\ {\bf 0}
\end{pmatrix}\!,
\]}
has a unique solution that satisfies
\begin{align}
\nonumber
\triple{ \bf E}_{|s|,\Omega_-} & +
   \|{\bf j} \|_{ -1/2, {\rm div}_{\Gamma}, \Gamma} +  \| {\bf m} \|_{-1/2 ,{\rm curl}_{\Gamma}, \Gamma} \\
\label{eq:LaplaceDomainStability}
  \lesssim \,& \frac{|s|^3} {\sigma \underline{\sigma}^4} \left( \| {\bf J} \|_{{\bf L}^2(\Omega_-)} +
   \|  \gamma_t^+{\bf \tiny curl\,E}^{inc}\| _{ -1/2, {\rm div}_{\Gamma}, \Gamma}  +   \| \pi_t^+{\bf E}^{inc} \|_{-1/2 ,{\rm curl}_{\Gamma}, \Gamma}  \right)  . 
\end{align}
Furthermore, if
\[
\mathbf E^{scat} = \widetilde{\mathcal D}(s/c_0)\mathbf m - \frac{1}{s\epsilon_0}\mathcal S(s/c_0)\mathbf j,
\]
then
{\small\begin{equation}
\label{eq:h5.9}
 \triple{ {\bf E}}_{|s|, \Omega_-} \!+ \! \triple{ {\bf E}^{scat}}_{|s|, \, \mathbb{R}^3 \setminus \Gamma}  \!\lesssim\!  \frac{|s|^3} {\sigma \underline{\sigma}^3} \! \left( \| {\bf J} \|_{{\bf L}^2(\Omega_-)}  \!+\!
   \|  \gamma_t^+{\bf curl\,E}^{inc}\| _{ -1/2, {\rm div}_{\Gamma}, \Gamma} \! +   \| \pi_t^+{\bf E}^{inc} \|_{-1/2 ,{\rm curl}_{\Gamma}, \Gamma}  \right)   
\end{equation}}
\end{theorem}
\begin{proof}
Unique solvability follows from Proposition \ref{prop:6.1} and the equivalence between problems \eqref{eq:h4.5} and \eqref{eq:variationalprob} established in Proposition \ref{prop:EquivalentBVP}. The stability estimate \eqref{eq:h5.9} follows readily from \eqref{eq:stabilityEstimate1} by setting $f$ equal to the linear form $\ell$ defined in \eqref{eq:variationalprobLinear} and ${\bf d} = \pi_t^+{\bf E}^{inc}$, and recalling that 
\[
\|\ell\| \lesssim\frac{|s|}{\underline\sigma} \left( \| {\bf J} \|_{{\bf L}^2(\Omega_-)}  \!+\!
   \|  \gamma_t^+{\bf curl\,E}^{inc}\| _{ -1/2,\, {\rm div}_{\Gamma}, \Gamma}\right).
\]

To obtain \eqref{eq:LaplaceDomainStability} we will make use of standard trace inequalities  (see e.g. Hsiao and Wendland \cite[5.1.11]{HsWe:2021} and Sayas \cite[A.3.2]{Sa:2016}) along with equation \eqref{eq:h3.17} to compute
\begin{align*}
 \|\, {\bf j} \|_{-1/2, \, {\rm div}_{\Gamma} , \Gamma} 
 & = \left\|\;  \jump{\, (s\mu_0)^{-1} \gamma_t {\bf curl\, E}^{scat}}\; \right\|_{-1/2, \, {\rm div}_{\Gamma},  \Gamma}  \\
&\lesssim  \frac{1}{|s|}\left(\| {\bf curl\,E}^{scat}\|_{{\bf L}^2(\mathbb{R}^3 \setminus \Gamma)} +  \| {\bf curl \, curl \,E}^{scat}\|_{ {\bf  L}^2(\mathbb{R}^3 \setminus \Gamma)}\right)
\\
& \lesssim\frac{1}{|s|}\left( \| {\bf curl\,E}^{scat}\|_{{\bf L}^2(\mathbb{R}^3 \setminus \Gamma)} + \frac{|s|^2}{c^2} \| {\bf E}^{scat}\|_{ {\bf L}^2(\mathbb{R}^3 \setminus \Gamma)}\right) \\
& \lesssim \frac{\max \{ 1, |s|\}}{|s|} \triple{ {\bf E}^{scat}}_{|s|,\,  \mathbb{R}^3 \setminus \Gamma}  \\
&\lesssim  \frac{1}{\underline{\sigma}} \triple{ {\bf E}^{scat}}_{|s|,\,  \mathbb{R}^3 \setminus \Gamma} \\[2ex] 
\| {\bf m} \|_{-1/2, \, {\rm curl}_{\Gamma} ,  \Gamma} &  = \left\| \; \jump{ \pi_t{\bf  E}^{scat}}\right\|_{-1/2,\,  {\rm curl}_{\Gamma}, \, \Gamma}  \lesssim \triple{{\bf E}^{scat}}_{1, \mathbb{R}^3 \setminus \Gamma} \lesssim  \frac{1}{\underline{\sigma}} \, \triple{ {\bf E}^{scat}}_{|s|, \mathbb{R}^3 \setminus \Gamma}.
\end{align*}  
Therefore, the two inequalities above imply that
\[
\triple{\bf E}_{|s|,\Omega_-} +  \|\, {\bf j} \|_{-1/2, \, {\rm div}_{\Gamma} , \Gamma}  + \| {\bf m} \|_{-1/2, \, {\rm curl}_{\Gamma} ,  \Gamma} \lesssim \frac{1}{\underline\sigma}\left( \triple{\bf E}_{|s|,\Omega_-} + \triple{{\bf E}^{scat}}_{|s|,\mathbb R^3\setminus\Gamma}\right)
\]
which, combined with \eqref{eq:h5.9}, leads to \eqref{eq:LaplaceDomainStability}.
\end{proof}
\begin{corollary}
\label{cor:cor6.1}
Let $\mathcal A(s)$ be the boundary-field operator implicitly defined by the left hand side of \eqref{eq:h4.5} and $\mathcal B(s): \mathbb H \to \mathbb H^\prime$ the operator associated to the right hand side of \eqref{eq:variationalprobBilinear}.
Both $\mathcal A(s)$ and $\mathcal B(s)$ are invertible and

\begin{equation}
\label{eq:InverseNorms}
\left\|\mathcal A^{-1}(s)\right\| \lesssim \frac{|s|^3}{\sigma\underline\sigma^5} \qquad \text{ and } \qquad
\left\|\mathcal B^{-1}(s)\right\| \lesssim \frac{|s|^3}{\sigma\underline\sigma^4}.
\end{equation}
\end{corollary}

\paragraph{Remark 1.} The reader familiar with Laplace-domain analysis techniques for the analysis of wave propagation, and estimates associated to conormal derivatives of wave functions,  may have expected a bound of the form $\| {\bf j} \|_{-1/2, \, {\rm div}_{\Gamma} , \Gamma}  \lesssim  \sqrt{\frac{ |s|}{\underline{\sigma}}} \triple{{\bf E}^{scat}}_{|s|,\,  \mathbb{R}^3 \setminus \Gamma}$, instead of the factor  $\frac{1}{\underline{\sigma}}$ obtained above. The reason for this is the lack of an electromagnetic analogous to Bamberger and Ha-Duong's optimal lifting for acoustic waves \cite{BaHa:1986a, BaHa:1986b}.

\color{black}
\paragraph{Remark 2.} In all the arguments above we have assumed that only electric field information is provided as problem data. However if additionally the incident magnetic field ${\bf H}^{scat}$ is known, then the estimates from Theorem \ref{thm:6.1} and Corollary \ref{cor:cor6.1} can be improved. This comes as a simple consequence of the fact that ${\bf H}^{inc} = (s\mu_0)^{-1}{\bf curl\, E}^{inc}$ and thus $\|\ell\| \lesssim |s|\left( \| {\bf J} \|_{{\bf L}^2(\Omega_-)}  \!+\!
   \|  \gamma_t^+{\bf H}^{inc}\| _{ -1/2 {\rm div}_{\Gamma}, \Gamma}\right)$, leading to the improved estimates below.
\begin{corollary}
\label{cor:cor6.2}
Given problem data 
\[
( {\bf J}, \gamma_t ^+{\bf H}^{inc},\pi^+_t{\bf E}^{inc},    ) \in \mathbf L^2(\Omega_-)\times {\bf H}^{-1/2}_{\|} ( {\rm div}_{\Gamma}, \Gamma) \times {\bf H}^{-1/2}_{\perp} ( {\rm curl}_{\Gamma}, \Gamma),
\] 
The unique solution triplet to the boundary-field formulation \eqref{eq:h4.5} satisfies the stability estimate
\begin{align*}
\triple{ \bf E}_{|s|,\Omega_-} & +
   \|{\bf j} \|_{ -1/2, {\rm div}_{\Gamma}, \Gamma} +  \| {\bf m} \|_{-1/2 ,{\rm curl}_{\Gamma}, \Gamma} \\
  \lesssim \,& \frac{|s|^2} {\sigma \underline{\sigma}^2} \left( \| {\bf J} \|_{{\bf L}^2(\Omega_-)} +
   \|  \gamma_t^+{\bf H}^{inc}\| _{ -1/2, {\rm div}_{\Gamma}, \Gamma}  +   \| \pi_t^+{\bf\,E}^{inc} \|_{-1/2 ,{\rm curl}_{\Gamma}, \Gamma}  \right). 
\end{align*}
If ${\bf E}^{scat}$ is defined through the electromagnetic layer potentials as in \eqref{eq:h4.7} then
{\small\[
 \triple{ {\bf E}}_{|s|, \Omega_-} \!+ \! \triple{ {\bf E}^{scat}}_{|s|, \, \mathbb{R}^3 \setminus \Gamma}  \!\lesssim\!  \frac{|s|^2} {\sigma \underline{\sigma}} \! \left( \| {\bf J} \|_{{\bf L}^2(\Omega_-)}  \!+\!
   \|  \gamma_t^+{\bf H}^{inc}\| _{ -1/2, {\rm div}_{\Gamma}, \Gamma} \! +   \| \pi_t^+{\bf E}^{inc} \|_{-1/2 ,{\rm curl}_{\Gamma}, \Gamma}  \right).   
\]}
Finally, the norms of the inverse operators $\mathcal A^{-1}(s)$ and $\mathcal B^{-1}(s)$ satisfy the estimates
\[
\left\|\mathcal A^{-1}(s)\right\| \lesssim \frac{|s|^2}{\sigma\underline\sigma^3} \qquad \text{ and } \qquad
\left\|\mathcal B^{-1}(s)\right\| \lesssim \frac{|s|^2}{\sigma\underline\sigma^2}.
\]
\end{corollary}
\color{black}
%
% ===============================================
\section{Results in the time-domain}\label{sec:7}
% ===============================================
%
Having established the properties of the operators and solutions to our problem in the Laplace domain, we can now return to the time domain and establish analogue results. Since this section deals only with time-domain results, we will use the convention that all variables are time-domain functions whose Laplace transforms were represented in the previous sections \textit{using the same character}. 
 
We will make use of results due to Laliena and Sayas \cite{LaSa:2009,Sa:2016} to transform our previous analysis into time domain statements. In order to state the result that will allow us to transfer our previous analysis we will first define a class of admissible symbols \cite[p.91]{Me:1983},  \cite{LaSa:2009}.

\begin{definition}[\bf A class of  admissible symbols]\label{def:6.1}
Let  $\mathbb X$ and $\mathbb Y$ be Banach spaces and  $\mathcal{B}(\mathbb X, \mathbb Y)$ be the set of bounded linear operators from $\mathbb X$ to $\mathbb Y$. An operator-valued analytic function $A : \mathbb{C}_+ \rightarrow \mathcal{B}(\mathbb X, \mathbb Y)$ is said to belong to the class $ \mathcal{E} (\theta, \mathcal{B}(\mathbb X, \mathbb Y))$, if there exists a real number $\theta$ such that 
\[
\|A(s)\|_{\mathbb X,\mathbb Y} \le C_A\left(\mathrm{Re} (s)\right) |s|^{\theta} \quad \mbox{for}\quad s \in \mathbb{C}_+ ,
\]
where the function $C_A : (0, \infty) \rightarrow (0, \infty) $ is non-increasing and satisfies 
\[
C_A(\sigma) \le \frac{ c}{\sigma^m} , \quad \forall \quad \sigma \in ( 0, 1]
\]
for some  $m \ge 0$ and $c$ independent of $\sigma$. 
\end{definition} 
The following proposition will be the key to transform the Laplace-domain bounds into time-domain statements.
\begin{proposition} {\em\cite{Sa:2016errata}} \label{pr:6.1}
Let $A = \mathcal{L}\{a\} \in \mathcal{E} (k + \alpha, \mathcal{B}(\mathbb X,\mathbb Y))$ with $\alpha\in [0, 1)$ and $k$ a non-negative integer.  If $ g \in \mathcal{C}^{k+1}(\mathbb{R}, \mathbb X)$ is causal and its derivative $g^{(k+2)}$ is integrable, then $a* g \in \mathcal{C}(\mathbb{R}, \mathbb Y)$ is causal and 
\[
\| (a*g)(t) \|_{\mathbb Y} \le 2^{\alpha} C_{\epsilon} (t) C_A (t^{-1}) \int_0^t \|(\mathcal{P}_2g^{(k)})(\tau) \|_{\mathbb X} \; d\tau,
\]
where 
\[
C_{\epsilon} (t) := \frac{1}{2\sqrt{\pi}} \frac{\Gamma(\epsilon/2)}{\Gamma\left( (\epsilon+1)/2 \right) } \frac{t^{\epsilon}}{(1+ t)^{\epsilon}}, \qquad (\epsilon :=  1- \alpha \; \;  \mbox{and}\; \;  \theta = k +\alpha)
\]
and 
\[
(\mathcal{P}_2g) (t) =  g + 2\dot{g} + \ddot{g}.
\]
\end{proposition}
We now define the spaces
\begin{alignat*}{6}
\mathbb X : =\,& {\bf L}^2 (\Omega_-)\times {\bf H}^{-1/2}_{\|}( {\rm div}_{\Gamma}, \Gamma) \times {\bf H}_{\perp}^{-1/2}( {\rm curl}_{\Gamma}, \Gamma), \\
\mathbb Y_1 :=\,&  {\bf H}({\bf curl},\Omega_-)\times \mathbf H^*\,, \\
\mathbb Y_2 :=\,& {\bf H}({\bf curl},\Omega_-)\times {\bf H}^{-1/2}_{\|}( {\rm div}_{\Gamma}, \Gamma)\times {\bf H}_{\perp}^{-1/2}( {\rm curl}_{\Gamma}, \Gamma),
\end{alignat*}
and apply Proposition \ref{pr:6.1} to Theorem \ref{thm:6.1} from the previous section.
\begin{theorem}
\label{th:7.1}
Given Laplace domain problem data $({ \bf J} , \gamma_t^+{\bf curl\, E}^{inc}, \pi_t^+{\bf E}^{inc}) \in \mathbb X$,
let
\[
{\bf D} (t) := \mathcal{L}^{-1} \left\{({ \bf J} , \gamma_t^+{\bf curl\, E}^{inc}, \pi_t^+{\bf E}^{inc}) \right\}
\]
be the corresponding time-domain data. If ${\bf D} \in  \mathcal{C}^{(4)} \left([0, T],  \mathbb X \right)$ is causal, and $ {\bf D}^{(5)}$ is integrable, then $( {\bf E}, {\bf E}^{scat})\in \mathcal{C} \left([0, T], \mathbb Y_1 \right)$ and 
\[
\|  ( {\bf E}, {\bf E}^{scat})(t) \|_{\mathbb{Y}_1}   \lesssim  \frac{t^2}{1 + t} \max\{1, t^4 \}\int_0^t \|(\mathcal{P}_2{\bf D}^{(3)})(\tau) \|_{ \mathbb{X} } \; d\tau.
\]
\end{theorem}
\begin{proof}
From \eqref{eq:InverseNorms} we see that $ \mathcal {B}^{-1}\in   \mathcal{E} (3, \mathcal{B}(\mathbb X, \mathbb Y))$.  Hence $k= 3, \alpha =0, \varepsilon = 1- \alpha =1$ and  finally $C_{\mathcal{B}^{-1} }(t^{-1}) = t \max\{ 1, t^4 \}$.
We note that
\[
( {\bf E}, {\bf E}^{scat})(t)  =  \left( \mathcal{L}^{-1} \{ {\mathcal B}^{-1}(s) \}  * {\bf D} \right)(t)
\]
which completes the proof.
\end{proof} 

Analogously, from the estimate for $\mathcal A^{-1}(s)$ in Corollary \ref{cor:cor6.1}, we have $k  = 3, \alpha =0, \varepsilon = 1$ and $ C_{ ^{-1}} (t^{-1}) =  t \max\{1, t^5\} $. Hence a similar argument proves that
\begin{theorem}
\label{th:7.2}
Given   $({ \bf J} , \gamma_t^+{\bf curl\, E}^{inc}, \pi_t^+{\bf E}^{inc}) \in \mathbb X$, if
\[
{\bf D}(t) := \mathcal{L}^{-1} \left\{({ \bf J} , \gamma_t^+{\bf curl\, E}^{inc}, \pi_t^+{\bf E}^{inc}) \right\} \in  \mathcal{C}^{(4)} ([0, T],  \mathbb{X})
\]
is causal, and $ {\bf D}^{(5)}$ is integrable, then  $( {\bf E, j, m } )$ belongs to $\mathcal{C}( [0, T], \mathbb Y_2) $ 
and 
\[
\|  ( {\bf E, j, m } ) (t) \|_{\mathbb Y_2}   \lesssim  \frac{t^2}{1 + t} \max\{1, t^5\} \int_0^t \|(\mathcal{P}_2{\bf D}^{(3)})(\tau) \|_{ \mathbb{X} } \; d\tau.
\]
\end{theorem} 

\color{black}Since the equations at hand are linear, that the discretization error will satisfy an analogous system with source terms related to the approximation properties of the discrete spaces utilized. This can be used to tanslate these estimates into semi-discrete and discrete error estimates for Convolution Quadrature. This follows from the fact that the estimates above can be directly applied to the transmitted and scattered error functions 
\[
{\bf e} : = {\bf E}-{\bf E}_h\,, \qquad {\bf e}^{scat} : = {\bf E}^{scat}-{\bf E}_h^{scat}\,, \qquad {\bf e_{j}} : = {\bf j} - {\bf j}_h\,, \qquad \text{ and } \qquad {\bf e_{m}} : = {\bf m} - {\bf m}_h\,,
\]
where the subscript $h$ is used to denote a semi-discrete Galerkin discretization in space. The process is outlined in detail in \cite[Section 5]{Sanchez-Vizuet:2024} for the case of a  boundary integral formulation for both transmitted and scattered fields.
 
\color{black}
%
% =========================================
\section{The case of coated conductors}\label{sec:8}
% =========================================
%
\begin{figure}[tb] \centering 
\includegraphics[width=0.4\linewidth]{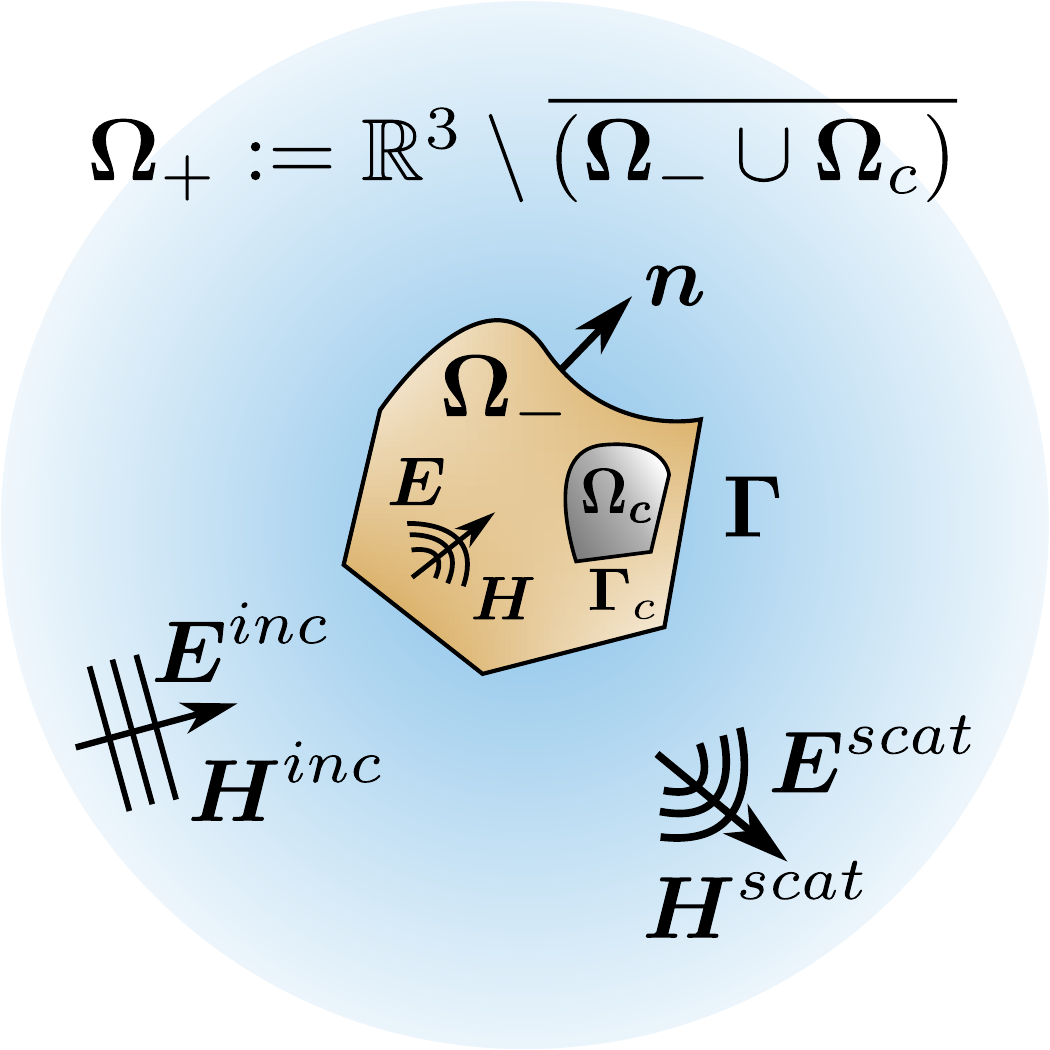}
\caption{A conducting material, contained in the region $\Omega_c$ with boundary $\Gamma_c$, is enclosed by a dielectric in the sorrounding region $\Omega_-$. The boundary of $\Omega_-$ has two disjoint components: the vacuum-dielectric interface $\Gamma$ and the dielectric-conductor interface $\Gamma_c$. The unit vector $\boldsymbol{n}$ anchored on $\Gamma_d$ and is exterior to $\Omega_-$.}\label{fig:fig2} 
\end{figure} 

The analysis toolbox that we have developped in the previous sections can be trivially extended to analyze the relevant case when the scatterer is a perfect conductor coated by a dielectric material. Since the proof process and the arguments used require only a simple modification with respect to the ones that we have used thus far, we will not go into full detail. Instead, we will point only to those details that require a small adaptation.

In this setting, depicted schematically in Figure \ref{fig:fig2}, a perfect conductor contained in an open domain $\Omega_c$ with Lipschitz boundary $\Gamma_c$ is enclosed by a bounded region $\Omega_-$ with Lipschitz boundary $\Gamma_d$. The domain $\Omega_-$ is occupied with a dielectric material and we will consider that the complementary unbounded region $\Omega_+:=\overline{\Omega_-\cup \Omega_c}$ is void, denoting by $\Gamma$ the inteface between the vacuum and dielectric media. Hence, the boundary of the dielectric can be decomposed in two disjoint components as $\Gamma_d = \Gamma \cup \Gamma_c$. The vacuum-dilectric interface $\Gamma$ coincides with its counterpart from the previous sections, while the disjoint component $\Gamma_c$ lies along the interface between the dielectric and the conductor. The assumption of $\Gamma_d$ being Lipschitz implies that the minimum distance between its two components $\Gamma$ and $\Gamma_c$ is positive.

The physical model in the vacuum region and inside the dielectric remains unchanged, as is the case with the transmission conditions in the vacuum-dielectric inteface $\Gamma$. On the other hand, inside the perfectly conducting region $\Omega_c$ the electric field must vanish identically. This, together with the conservation of the tangential component of the  electric field, implies that the tangential trace of the field ${\bf E}$ must vanish, leading to the system of time-domain governing equations
\begin{alignat*}{6}
{\bf curl} \,{\bf curl}\,  {\bf E} + c^{-2} \, \partial^2_{tt}\,{\bf E}   = \, & - \mu \partial_t {\bf J} & \quad & \text{ in } \quad \Omega_- \times (0, T), \\
 {\bf curl}\, {\bf curl}\, {\bf E}^{scat} + c_0^{-2} \, \partial^2_{tt} \,{\bf E}^{scat}  =\, &  {\bf 0} & \quad & \text{ in } \quad \Omega_+ \times (0, T), \\
 { \bf n} \times ({ \bf E} - {\bf E}^{scat})\times {\bf n}  =\, & {\bf n} \times {\bf E}^{inc}\times {\bf n} & \quad & \text{ on } \quad \Gamma\times (0,T),  \\
 {\bf n} \times { \bf curl} \, ( \mu^{-1}{\bf E} - \mu_0^{-1}{\bf E}^{scat}) =\, &  {\bf n} \times \mu_0^{-1}{\bf curl }\, {\bf E}^{inc} & \quad & \text{ on } \quad \Gamma \times (0, T),\\ 
{\bf n} \times \mu^{-1}{\bf E}\times {\bf n} =\, & {\bf 0} & \quad & \text{ on } \quad \Gamma_c \times (0, T), 
\end{alignat*}
together with initial conditions
\begin{alignat*}{8}
{\bf E}(x, 0) =&\, {\bf 0},  &\quad &  \partial_t {\bf E} (x, 0)  &=\, {\bf 0}  & \quad &  \text{ in }\, \Omega_-\,, \\
{\bf E}^{scat}(x, 0) =&\,  {\bf 0},  &\quad &  \partial_t {\bf E}^{scat}(x, 0) &=\, {\bf 0} & \quad &  \text{ in }\, \Omega_+\,.
\end{alignat*}
Just as before, assuming that ${\bf E}^{inc}$ is supported away from $\overline{\Omega_-\cup\Omega_c}$ initially, and that both ${\bf E}$ and ${\bf E}^{scat}$ are causal allows us to transform the time domain system into its Laplace domain-counterpart
\begin{subequations}
\label{eq:coatedconductor}
\begin{alignat}{6}
\label{eq:coatedconductorA}
{\bf curl} \,{\bf curl}\,  {\bf E} + (s/c)^2{\bf E}   = \, & - \mu s {\bf J} & \quad & \text{ in } \quad \Omega_- , \\
\label{eq:coatedconductorB}
 {\bf curl}\, {\bf curl}\, {\bf E}^{scat} + (s/c_0)^2 {\bf E}^{scat}  =\, &  {\bf 0} & \quad & \text{ in } \quad \Omega_+ , \\
 \label{eq:coatedconductorC}
 \pi_t^-{ \bf E} - \pi_t^+{\bf E}^{scat}  =\, & \pi_t^+{\bf E}^{inc} & \quad & \text{ on }\quad  \Gamma,  \\
 \label{eq:coatedconductorD}
 \mu^{-1}\gamma_t^-{ \bf curl \, E} - \mu_0^{-1}\gamma_t^+{\bf curl\, E}^{scat} =\, & \mu_0^{-1}\gamma_t^+{\bf curl }\, {\bf E}^{inc} & \quad & \text{ on } \quad \Gamma,\\
 \label{eq:coatedconductorE} 
\mu^{-1}\pi_t^-{\bf E} =\, & {\bf 0} & \quad & \text{ on } \quad \Gamma_c. 
\end{alignat}
\end{subequations}
 Just as in the previous sections, and for the sake of simplicity, when transforming the system into the Laplace domain we have used the convention that Laplace domain and time domain functions are represented by the same symbol. 
 
The system \eqref{eq:coatedconductor} above is essentially equal to \eqref{eq:h2.8} with the addition of the extra boundary component $\Gamma_c$, where a homogeneous boundary condition for the transmitted field ${\bf E}$ has been prescribed. Hence, we need only to modify the functional space for ${\bf E}$ in order to account for this additional essential boundary condition. This results in 
\[
{\bf H_c}({\bf curl},\Omega_-):=\left\{{\bf u}\in {\bf H}({\bf curl},\Omega_-) : \pi^-{\bf u} = {\bf 0} \;\;\text{ on }\;\;\Gamma_c  \right\}.
\] 
Replacing then every instance of the space ${\bf H}({\bf curl}, \Omega_-)$ by the space by ${\bf H_c}$ in Sections \ref{sec:4}, \ref{sec:5}, \ref{sec:6}, and \ref{sec:7}, all the results proven carry over verbatim to the coated conductor case.
%
% ==============================
\section*{Acknowledgements}
% ==============================
%
Tonatiuh S\'anchez-Vizuet has been partially funded by the U. S. National Science Foundation through the grant NSF-DMS-2137305.

% ==========================
{\footnotesize
%\newpage
%\bibliographystyle{plain}
\bibliographystyle{abbrv}
\bibliography{references}
}
% ==========================

\end{document}